\documentclass[11pt,a4paper]{article}
\usepackage{amsmath}
\usepackage{amsthm}
\usepackage{amssymb}
\usepackage{amscd}
\usepackage{graphicx}
\usepackage{epsfig}
\usepackage[matrix,arrow,curve]{xy}
\usepackage{color}
\usepackage{mathrsfs}
\usepackage[scr=rsfs,cal=boondox]{mathalpha}

\usepackage{bbm}
\usepackage{stmaryrd}
\usepackage{hyperref}

\usepackage{latexsym}
\usepackage{amsfonts}
\input xy

\newtheorem{theorem}{Theorem}[section]
\newtheorem{proposition}[theorem]{Proposition}

\theoremstyle{definition}
\newtheorem{example}[theorem]{Example}
\newtheorem{remark}[theorem]{Remark}
\newtheorem{definition}[theorem]{Definition}

\font\black=cmbx10 \font\sblack=cmbx7 \font\ssblack=cmbx5 \font\blackital=cmmib10  \skewchar\blackital='177
\font\sblackital=cmmib7 \skewchar\sblackital='177 \font\ssblackital=cmmib5 \skewchar\ssblackital='177
\font\sanss=cmss11 \font\ssanss=cmss8 scaled 900 \font\sssanss=cmss8 scaled 600 \font\blackboard=msbm10
\font\sblackboard=msbm7 \font\ssblackboard=msbm5 \font\caligr=eusm10 \font\scaligr=eusm7 \font\sscaligr=eusm5

\font\bsymb=cmsy10 scaled\magstep2
\def\all#1{\setbox0=\hbox{\lower1.5pt\hbox{\bsymb
       \char"38}}\setbox1=\hbox{$_{#1}$} \box0\lower2pt\box1\;}
\def\exi#1{\setbox0=\hbox{\lower1.5pt\hbox{\bsymb \char"39}}
       \setbox1=\hbox{$_{#1}$} \box0\lower2pt\box1\;}

\def\tx#1{{\fam0\relax#1}}

\newfam\bifam
\textfont\bifam=\blackital \scriptfont\bifam=\sblackital \scriptscriptfont\bifam=\ssblackital

\newfam\blfam
\textfont\blfam=\black \scriptfont\blfam=\sblack \scriptscriptfont\blfam=\ssblack

\newfam\bbfam
\textfont\bbfam=\blackboard \scriptfont\bbfam=\sblackboard \scriptscriptfont\bbfam=\ssblackboard

\newfam\ssfam
\textfont\ssfam=\sanss \scriptfont\ssfam=\ssanss \scriptscriptfont\ssfam=\sssanss
\def\sss#1{{\fam\ssfam\relax#1}}

\newfam\clfam
\textfont\clfam=\caligr \scriptfont\clfam=\scaligr \scriptscriptfont\clfam=\sscaligr

\def\pmb#1{\setbox0\hbox{${#1}$} \copy0 \kern-\wd0 \kern.2pt \box0}
\def\pmbb#1{\setbox0\hbox{${#1}$} \copy0 \kern-\wd0
      \kern.2pt \copy0 \kern-\wd0 \kern.2pt \box0}
\def\pmbbb#1{\setbox0\hbox{${#1}$} \copy0 \kern-\wd0
      \kern.2pt \copy0 \kern-\wd0 \kern.2pt
    \copy0 \kern-\wd0 \kern.2pt \box0}
\def\pmxb#1{\setbox0\hbox{${#1}$} \copy0 \kern-\wd0
      \kern.2pt \copy0 \kern-\wd0 \kern.2pt
      \copy0 \kern-\wd0 \kern.2pt \copy0 \kern-\wd0 \kern.2pt \box0}
\def\pmxbb#1{\setbox0\hbox{${#1}$} \copy0 \kern-\wd0 \kern.2pt
      \copy0 \kern-\wd0 \kern.2pt
      \copy0 \kern-\wd0 \kern.2pt \copy0 \kern-\wd0 \kern.2pt
      \copy0 \kern-\wd0 \kern.2pt \box0}

\mathchardef\za="710B  
\mathchardef\zb="710C  
\mathchardef\zg="710D  
\mathchardef\zd="710E  
\mathchardef\zve="710F 
\mathchardef\zz="7110  
\mathchardef\zh="7111  
\mathchardef\zvy="7112 
\mathchardef\zi="7113  
\mathchardef\zk="7114  
\mathchardef\zl="7115  
\mathchardef\zm="7116  
\mathchardef\zn="7117  
\mathchardef\zx="7118  
\mathchardef\zp="7119  
\mathchardef\zr="711A  
\mathchardef\zs="711B  
\mathchardef\zt="711C  
\mathchardef\zu="711D  
\mathchardef\zvf="711E 
\mathchardef\zq="711F  
\mathchardef\zc="7120  
\mathchardef\zw="7121  
\mathchardef\ze="7122  
\mathchardef\zy="7123  
\mathchardef\zf="7124  
\mathchardef\zvr="7125 
\mathchardef\zvs="7126 
\mathchardef\zf="7127  
\mathchardef\zG="7000  
\mathchardef\zD="7001  
\mathchardef\zY="7002  
\mathchardef\zL="7003  
\mathchardef\zX="7004  
\mathchardef\zP="7005  
\mathchardef\zS="7006  
\mathchardef\zU="7007  
\mathchardef\zF="7008  
\mathchardef\zW="700A  

\newcommand{\be}{\begin{equation}}
\newcommand{\ee}{\end{equation}}
\newcommand{\raa}{\rightarrow}

\newcommand{\bea}{\begin{eqnarray}}
\newcommand{\eea}{\end{eqnarray}}
\newcommand{\beas}{\begin{eqnarray*}}
\newcommand{\eeas}{\end{eqnarray*}}
\def\*{{\textstyle *}}

\newcommand{\ot}{\otimes}

\newcommand{\pa}{\partial}
\newcommand{\ti}{\times}

\newcommand{\ad}{{\rm ad}}

\newcommand{\Ll}{{\pounds}}

\def\ran{\rangle}

\def\cL{{\mathcal L}}

\def\cR{{\mathcal R}}

\def\cO{{\mathcal O}}

\def\wt{\widetilde}

\def\ul{\underline}

\def\sJ{{\sss J}}

\def\sT{{\sss T}}
\def\sV{{\sss V}}

\def\sj{{\sss j}}

\def\xd{\tx{d}}

\newdir{|>}{%
!/4.5pt/@{|}*:(1,-.2)@^{>}*:(1,+.2)@_{>}}

\newcommand{\la}{\langle}

\newcommand{\Z}{\mathbb{Z}}
\newcommand{\R}{\mathbb{R}}

\newcommand{\Pe}{\mathbb{P}}

\newcommand{\n}{\nabla}

\newcommand{\op}[1]{\!\!\mathop{\rm ~#1}\nolimits}

\newcommand{\id}{\op{id}}

\def\Rt{{\R^\ti}}
\def\Lt{{L^\ti}}

\def\VB{{\rm V\!B}}

\DeclareMathOperator{\GL}{GL}

\newcommand{\we}{\wedge}

\newcommand{\p}{\mathbf{p}}

\begin{document}
\title{\bf A geometric approach to contact Hamiltonians \\ and contact Hamilton-Jacobi theory}
\date{}
\author{\\ Katarzyna  Grabowska$^1$\\ Janusz Grabowski$^2$ 
        \\ \\
         $^1$ {\it Faculty of Physics}\\
                {\it University of Warsaw}\\
                \\$^2$ {\it Institute of Mathematics}\\
                {\it Polish Academy of Sciences}
                }
\maketitle
\begin{abstract}
\noindent We propose a novel approach to contact Hamiltonian mechanics which, in contrast to the one dominating in the literature, serves also for non-trivial contact structures. In this approach Hamiltonians are no longer functions on the contact manifold $M$ itself but sections of a line bundle over $M$ or, equivalently, 1-homogeneous functions on a certain $\GL(1,\R)$-principal bundle $\zt:P\to M$, which is equipped with a homogeneous symplectic form $\zw$. In other words, our understanding of contact geometry is that it is not an `odd-dimensional cousin' of symplectic geometry but rather a part of the latter, namely `homogeneous symplectic geometry'. This understanding of contact structures is much simpler than the traditional one and very effective in applications, reducing the contact Hamiltonian formalism to the standard symplectic picture. We develop in this language contact Hamiltonian mechanics in the autonomous, as well as the time-dependent case, and the corresponding Hamilton-Jacobi theory. Fundamental examples are based on canonical contact structures on  the first jet bundles $\sJ^1L$ of sections of line bundles $L$, which play in contact geometry a fundamental r\^ole, similar to that played by cotangent bundles in symplectic geometry.

\par

\end{abstract}


\section{Introduction}
The roots of contact geometry can be traced back to 1872, when Sophus Lie
introduced the concept of contact transformations to study systems of differential equations.
The contact structures reappeared in Gibbs’ work on thermodynamics, Huygens’ work
on geometric optics, Hamiltonian dynamics, fluid mechanics, etc. As a nice  source for the  history of contact geometry and topology with an extended list of references we recommend \cite{Geiges:2001}, and for comprehensive presentation of contact geometry we refer to classical monographs \cite{Arnold:1989,Geiges:2008,Libermann:1987}.

More recently, there has been a revival of interest in the study of contact dynamics and their geometric properties, for instance in mesoscopic dynamics, thermodynamics (equilibrium and non-equilibrium), statistical mechanics, and study of systems with dissipation. The geometric approach assigns a contact structure to the thermodynamic phase space, such that the dynamics is represented by contact Hamiltonian  vector fields, and Legendre submanifolds describe equilibrium states. Especially, various applications in thermodynamics appeared in numerous papers, to mention Gibbs' book \cite{Gibbs:1948}, papers by Mruga{\l}a and collaborators \cite{Mrugala:1991,Mrugala:2000,Mrugala:1993}, and more recent papers
\cite{Baldiotti:2016,Bravetti:2019,Goto:2015,Grmela:2014,Rajeev:2008,
Simoes:2020,Simoes:2021,Shaft:2018}. In the latest works by Esen and his collaborators \cite{Esen:2022,Esen:2022a} the role of contact geometry in statistical mechanics and
thermodynamics is deeply studied, including fully geometric formulation of the General Equation for Non-Equilibrium Reversible-Irreversible Coupling (GENERIC). The authors discuss many concepts we also deal with in this paper.

Contact geometry has been revealed relevant in the last years also due to its applications to describe mechanical dissipative systems, both in the Hamiltonian and Lagrangian descriptions
(for instance, \cite{deLeon:2020,deLeon:2021,deLeon:2021a,Gaset:2020}), sometimes associated with the  Herglotz' generalized variational principle, but also quantum systems \cite{Ciaglia:2018}.
The standard Hamiltonian formulation describes exclusively isolated systems with reversible dynamics, while real systems are constantly in interaction with the environment, which introduces the phenomena of dissipation and irreversibility. Therefore a major question is whether it is possible to construct a classical mechanical theory that not only contains all the advantages of the Hamiltonian formalism, but also takes into account the effects of the environment on the system.
Hence, a major goal, at least in some areas of physics, is that of finding generalizations of Hamilton equations that apply to systems exchanging energy with the environment.
A lot of papers in this subject deal with contact  Hamiltonian mechanics, e.g. \cite{Bravetti:2017,Bravetti:2017a,deLeon:2017,deLeon:2019}, including contact Hamilton-Jacobi theory \cite{Cannarsa:2019,Cruz:2018,deLeon:2021c,Esen:2021a,Grillo:2020,Rajeev:2008}.

Unfortunately, almost all of them deal with trivial (cooriented) contact structures (a global contact form $\zh$ on a manifold $M$ is given) and are slight modifications of a single (and quite old) idea: the contact  dynamics associated with a `contact Hamiltonian' $H:M\to\R$ is represented by the contact Hamiltonian vector field $X^c_H$ determined uniquely by the equations
\be\label{E3}i_{X^c_{H}}\eta =-H,\qquad i_{X^c_{H}}\xd\eta =\xd H-{\cR_\zh}(H) \eta\,,\ee
where $\mathcal{R}_\zh$ is the so called \emph{Reeb vector field} for $\zh$. The problem is that
this definition of  the contact Hamiltonian vector field is strongly associated with the choice of the contact form  $\zh$ and the corresponding Reeb vector field, and does not make any sense for more general contact structures: equations (\ref{E3}) are not invariant with respect to the choice of $\zh$ among equivalent contact forms. The point is that changing a (local) trivialization we  have to change also the Hamiltonian. In this sense, such an approach does not serve even for trivializable (but not trivial) contact structures. In other words, equations (\ref{E3}) have no geometric sense in non-trivial contact geometry in which contact structures are `maximally-nonintegrable' distributions of hyperplanes on $M$. What is more, such an approach ignores many canonical and very important examples of contact structures which are not trivial, e.g. first jets $\sJ^1(L)$ of non-trivial line bundles or projectivized cotangent bundles $\Pe(\sT^*M)$ which are non-orientable for odd-dimensional $M$. On the other hand, even authors working  exclusively with globally defined contact forms (especially with  \emph{extended cotangent bundles} $\sT^*Q\ti\R$) accept \emph{contact transformations} as preserving the contact form up to a factor, that is a slight inconsequence.

Also the proposed Lagrangian formalisms in this context are usually defined only for regular `contact Lagrangians' or Lagrangians on \emph{extended tangent bundles} $\sT Q\ti\R$, except for some recent attempts \cite{deLeon:2021,Esen:2021,Esen:2021a,Liu:2018} to extend it for singular Lagrangians in the spirit of \emph{Tulczyjew triples}, invented by Tulczyjew \cite{Tulczyjew:1974,Tulczyjew:1977,Tulczyjew:1999} as an effective  geometrical tool to deal with singular Lagrangians. However, the corresponding contact Euler-Lagrange equations have connections to Hamilton-Jacobi theory and variational approach of Herglotz \cite{Herglotz:1930}, who used a generalization of the well-known Hamilton principle, that miraculously provides the same equations that we can obtain using contact geometry. Of course, all the above-mentioned papers form by no means a complete list  of references in the subject. Note that a full contact version of the Tulczyjew triples we proposed recently in \cite{Grabowska:2022a}.

\smallskip
In the present paper we propose a new picture for contact Hamiltonian mechanics and contact Hamilton-Jacobi theory, which is geometrically intrinsic and valid for general contact structures, not necessarily trivial (cooriented). The point is that contact Hamiltonians for a general contact structure on a manifold $M$ are not functions on $M$ but sections of a certain line bundle over $M$ and there is no way in non-trivial cases to associate with them functions (Hamiltonians) on $M$. This understanding of Hamiltonians as sections of line bundles serves actually for more general structures, so called \emph{Jacobi bundles} (but also \emph{local Lie algebras} or \emph{Kirillov manifolds}.

To be more precise, let us recall that in the standard description, a contact structure on a manifold $M$ of  the odd dimension $2n+1$ is viewed as certain `maximally non-integrable' one-codimensional subbundle $C$ in the tangent bundle $\sT M$ (a distribution of hyperplanes). The fibers of this subbundle are called \emph{contact elements} in the traditional language of contact geometry. The subbundle $C$ is locally the kernel of a non-vanishing 1-form $\zh$ on a manifold $M$ which induces a rank-one subbundle in $\sT^* M$. Maximal non-integrability of $C$ is expressed by the condition that the volume form $\zh\we(\xd\zh)^n$ is nowhere vanishing. Such 1-forms are called \emph{contact forms}. Two such forms are viewed \emph{equivalent} if they have the same kernel, i.e., they differ by a factor which is an invertible function; forms equivalent to contact forms are contact forms themselves. If a global contact form determining $C$ is fixed (such a form may not exist), the contact structure we call \emph{trivial} or \emph{cooriented}. Having a trivial contact structure $C$ with the global contact form $\zh$ on $M$, we can construct a symplectic structure $\zw_\zh$ on the manifold $\bar M=\R\ti M$,
$$\zw(s,x)=\xd(e^s\cdot\zh)(s,x)=e^s\cdot\xd s\we\zh(x)+e^s\cdot\xd\zh(x)\,.$$
This symplectic form is 1-homogeneous in the sense that $\Ll_{\pa_s}\zw_\zh=\zw_\zh$.
The symplectic manifold $(\bar M,\zw_\zh)$ is called the \emph{symplectization} of $(M,\zh)$.
For general contact structures symplectizations are generally more sophisticated.
Actually, in this paper the symplectization will be understood as a principal $\GL(1,\R)$-bundle $\zt:P\to M$ equipped with a symplectic form $\zw$ which is 1-homogeneous with respect to the $\GL(1,\R)$-action. Such an object we will call a \emph{symplectic principal $\R^\ti$-bundle}. Throughout the paper, for a vector bundle $E$ we will denote with $E^\ti$ the submanifold in $E$ of non-zero vectors. In particular,  instead of the general linear group in dimension 1, $\GL(1,\R)$, we will write simply $\Rt=\R\setminus\{ 0\}$. Our main observation to start with is the following.
\begin{theorem}
There is a canonical one-to-one correspondence between contact structures $C\subset\sT M$ on a manifold $M$  and symplectic $\Rt$-principal bundles $P$ over $M$. In this correspondence  the symplectic $\Rt$-principal bundle associated with $C$ can be identified with $(C^o)^\ti\subset\sT^*M$, where $C^o$ is the annihilator of $C$.
\end{theorem}
\noindent In consequence of the homogeneity of the symplectic form, 1-homogeneous functions on $P$ are closed with respect to the Poisson bracket $\{\,\cdot,\cdot\,\}_\zw$ associated with the symplectic form $\zw$. On the other hand, 1-homogeneous functions on $P$ may be identified with sections $\zs$ of a certain line bundle (vector bundle of rank 1) $\zt^*_P:L^*_P\to M$ and the Poisson bracket $\{\,\cdot,\cdot\,\}_\zw$ induces a \emph{Jacobi bracket} $\{\zs,\zs'\}^J_\zw$ on sections of $L^*_P$ which makes the whole structure a \emph{local Lie algebra} in the sense of Kirillov \cite{Kirillov:1976}, a \emph{Jacobi bundle} in the sense of Marle \cite{Marle:1991}, or a \emph{Kirillov manifold (Kirillov bracket)} \cite{Bruce:2017,Grabowski:2013}. In our terminology, the latter objects are \emph{Poisson $\Rt$-principal bundles}, and the Hamiltonian geometry can be naturally extended to them. A more general algebraic treatment of Lie brackets of these types one can find in \cite{Grabowski:2003a,Grabowski:2013a}. Unfortunately, this nice geometrically intrinsic approach is nearly not present in the literature on contact  Hamiltonian mechanics.
In our contact case, the function on $P$ associated with a section $\zs$ we will denote $\zi_\zs$.
Passing to a (local) trivialization $P=\Rt\ti M$ of the principal bundle $P$ we determine a (local) contact form $\zh$ generating the corresponding contact structure, for which $\zw=\xd s\we\zh+s\cdot\xd\zh$. Note that the use of the multiplicative group $\Rt$ (which is non-connected) instead of the additive group $\R$ is crucial for including non-trivial contact structures into the picture.\newline

Elements of this concept of the symplectization of a contact structure with concrete $(P,\zw)$ one may find in the literature \cite{Arnold:1989,Libermann:1987}. However, we decided to draw full consequences of these constructions and regard the symplectizations themselves as genuine contact structures \cite{Bruce:2017,Grabowski:2013}. In other words, contact structures on $M$ are in our setting just abstract symplectic principal $\R^\ti$-bundles $\zt:P\to M$. In  this sense, contact geometry is not an `odd dimensional version' of symplectic geometry but rather `homogeneous symplectic geometry'. This approach is very simple and, as we will show in this paper, has many advantages with respect to the standard treatment of contact manifolds. For instance, the celebrated contact (Legendre) bracket of functions on $M$ induced by a global contact form is just the symplectic Poisson bracket $\{\,\cdot,\cdot\,\}_\zw$ reduced to 1-homogeneous functions, Legendre submanifolds in $M$ can be interpreted just as $\Rt$-invariant Lagrangian submanifolds in $(P,\zw)$, and contact Hamiltonians $H$ are 1-homogeneous functions on $P$ (equivalently: sections of the line bundle $L^*_P$). The \emph{clou} is that the Hamiltonian vector fields $X_{\zi_\zs}$ associated with a 1-homogeneous Hamiltonians $H=\zi_\zs$ on the symplectic manifold $(P,\zm)$ are projectable onto $M$, so that it is precisely $X^c_\zs=\zt_*(X_{\zi_\zs})$ which is the contact Hamiltonian vector field on $M$. Only for trivial contact structures $P=\Rt\ti M$, 1-homogeneous Hamiltonians on $P$ are of the form $H(s,x)=s\cdot \ul{H}(x)$, where $\ul{H}$ is a function  (`contact Hamiltonian') on $M$, and we recover the contact dynamics (\ref{E3}). The nonautonomous (time-dependent) version of contact dynamics is based on the extension $\bar M=M\ti\sT^*\R$ of the contact manifold and the  corresponding extension $\bar P=P\ti\sT^*\R$ of the $\Rt$-principal bundle $P$, supplemented by the extension of the symplectic form $\bar\zw=\zw+\zw_{\R}$, where $\zw_{\R}=\xd t\we\xd  p$ is the canonical symplectic form on $\sT^*\R$. To  1-homogeneous time-dependent Hamiltonians $H$ on $P$  we associate homogeneous autonomous Hamiltonians $\bar H$ on $\bar  P$  and  proceed as in the autonomous case.
The starting point for the contact Hamilton-Jacobi theory is then the observation that the contact Hamiltonian vector field $X^c_\zs$ is tangent to a Legendre submanifold $\cL_0$ in $M$ if and only if the section $\zs$ vanishes on $\cL_0$.

\medskip
The fundamental models of contact structures are provided by the first jet bundles $\sJ^1(L)\to Q$ of sections of line bundles $L\to Q$: the canonical contact structure of $\sJ^1(L^*)$ is represented by $P=\sT^*(L^\ti)$, where $L^\ti$ is the $\Rt$-principal bundle of non-zero vectors in $L=(L^*)^*$. Such a contact structure is cooriented if and only if $L$ is trivial, $L=\R\ti Q$.
In \cite{Grabowski:2013} it was shown that any \emph{linear} contact structure is equivalent to the one of these canonical contact structures. This is a contact analog of the well-known fact that every linear symplectic structure, i.e., a 1-homogeneous symplectic structure on a vector bundle $E\to M$, is equivalent to the canonical symplectic structure on $\sT^*M$.

In the case $P=\sT^*(L^\ti)$, the 1-homogeneous Hamiltonians are interpreted as sections of the line  bundle $\sJ^1(L^*)\ti_QL^*$ over $\sJ^1(L^*)$, i.e., maps $\zs:\sJ^1(L^*)\to L^*$ covering the identity on $Q$.  Canonical  Legendre submanifolds in $\sJ^1(L^*)$ are of the  form $\cL_0(S)=\sj^1(S)(Q)$, where $S:Q\to L^*$ is a  section of $L^*$ and $\sj^1(S):Q\to\sJ^1(L^*)$ is the corresponding section of the jet bundle $\sJ^1(L^*)$.  The corresponding  Lagrangian  submanifolds  in  $P=\sT^*(\Lt)$ are of the standard  form $\cL(S)=(\xd\,\zi_S)(\sJ^1(L^*))$. The canonical projection $\sJ^1(L^*)\to Q$ maps diffeomorphically the Legendre submanifold $\cL_0(S)$ onto  $Q$, and  hence  the contact Hamiltonian vector field $X^c_\zs$ restricted to $\cL_0(S)$  onto a vector field $\ul{X}^c_\zs$  on $Q$. The  contact Jacobi Theorem states that the vector  fields  $\ul{X}^c_\zs$ and $X^c_\zs$  are $\sj^1(S)$-related if and only if $\zs$  vanishes on  $\cL_0(S)$. A compact  form of the \emph{contact Hamilton-Jacobi equation} is then the PDE
$$\zs\circ\sj^1(S)=0$$
in the autonomous case, and
$$\zs\bigg(\sj^1(S_t),t\bigg)+\frac{\pa S}{\pa t}=0$$
for time-depending sections. For particular Hamiltonians we recover some versions of Hamilton-Jacobi equations present already  in the literature, e.g. \emph{discounted Hamilton-Jacobi equations} \cite{Cannarsa:2020,Jin:2000}.

To sum up, in our paper we develop a geometrically intrinsic model for contact Hamiltonian dynamics, which generalizes majority of approaches to contact Hamiltonian mechanics present in the literature for the case of non-trivial contact structures. It is based on the understanding of contact structures as certain homogeneous symplectic structures, so that almost all symplectic methods present in the standard Hamiltonian mechanics can be easily translated to the contact case. In consequence, topological problems (like cohomology) become important in this setting, that agrees with recent trends in Mathematical Physics. We do not consider particular questions in Physics, providing just a tool for a mathematical formulation of such questions, so that physicists can freely choose a  proper Hamiltonian for their problems in Physics, which needs of course, as usual, a mathematical interpretation of physical quantities.

\medskip The structure of the paper is the following. In the next section we present relationships between  line and $\Rt$-principal bundles, which play a relevant r\^ole in the sequel. Section 3 contains a discussion about various concepts of a contact structure and some crucial examples, while the next two sections are devoted to the development of contact Hamiltonian formalism in the autonomous (Section 4), as well as in the time-dependent case (Section 5). The latter section contains Example \ref{FE} in which you can trace all subtleties of our approach. The corresponding contact Hamilton-Jacobi theory one can find in Section 6. We end up with concluding remarks containing perspectives of our further activity in the subject.

\section{Line and $\R^\ti$-principal bundles}
The well-known construction of vector bundles $E$ of rank $n$ over a manifold $M$ as associated with  $\GL(n,\R)$-principal bundles $P$ over $M$ and a $\GL(n,\R)$-action on $\R^n$ are particularly simple in the case $n=1$. The rank $1$ vector bundles we will call \emph{line bundles}. For any line bundle $\zt_0:L\to M$ the manifold $L^\ti=L\setminus 0_M$, i.e., the open-dense subset of $L$ remaining after the removing points of the zero-section (identified often with just $M$), is canonically a $\GL(1,\R)$-principal bundle. The action is just the multiplication of non-zero vectors of $L$ by non-zero reals. We will write consequently $\R^\ti$ for $\GL(1,\R)=\R\setminus\{ 0\}$, so that $P=L^\ti$ is canonically an $\Rt$-principal bundle over $M$ with the $\Rt$-action
$$h:\Rt\ti L^\ti\to L^\ti\,,\quad h_s(v)=h(s,v)=s\cdot v\,,$$
for any non-zero vector $v\in L$. The line bundle $L_P$ over $M$, associated with an $\Rt$-principal bundle $\zt:P\to M$ with the $\R^\ti$-action $h_s$, is the set of orbits (cosets) of the $\Rt$-principal bundle $P\ti\R$ (and simultaneously a vector bundle over $P$) with the action $\tilde h_s(v_x,a)=(s\cdot v_x,s^{-1}\cdot a)$.
It is easy to see that $L_{L^\ti}=L$, so that from the principal bundle $P=L^\ti$ we reconstruct $L$ as the associated vector bundle. Indeed, it is easy to see that the map
$$\zf:L^\ti\ti\R\to L\,,\quad \zf(v_x,a)=a\cdot v_x$$
is constant on orbits of $\tilde h$ and yields the identification $L_{L^\ti}\simeq L$.
The dual bundle $\zt_0:L^*\to M$ can be, in turn, identified with the vector bundle \emph{coassociated} with $P=L^\ti$, i.e., the vector bundle of orbits of the $\Rt$-action $\tilde h^*$ on $P\ti\R$,
$$\tilde h^*_s(v_x,a^*)=(s\cdot v_x,s\cdot a^*)\,.$$
The canonical pairing between cosets $[(v_x,a)]\in (L_P)_x$ and $[(v_x,a^*)]\in (L^*_P)_x$ gives $a\,a^*\in\R$.

\medskip
It is well known that $G$-principal bundles $\zt:P\to M$ are determined by their transformation data. More precisely, any family of local trivializations $\zf_\za:\zt^{-1}(U_\za)\to U_\za\ti  G$, associated with an open covering $\{ U_\za\}$ of $M$, induces transition functions $F_{\za\zb}:U_\za\cap U_\zb\to G$ such that on
$$\zf_\za\circ\zf_\zb^{-1}:(U_\za\cap U_\zb)\ti G\to(U_\za\cap U_\zb)\ti G\,,\quad \zf_\za\circ\zf_\zb^{-1}(x,g)=(x,F_{\za\zb}(x)\cdot g)\,,$$
and the collection $\{F_{\za\zb}\}$ (defining a \v{C}ech 1-cocycle) completely determines $P$. Note that the same transitions functions serve for the corresponding local trivializations of the associated vector bundle. An important observation in the case of $\Rt$-principal bundles (or line bundles) is that we can always find trivializations with the transition functions taking only values $\pm 1$. Indeed, a part of the $\Rt$-action on $P$ is a $\Z_2$-action of multiplication by $\pm 1$. The quotient $P'=P/\Z_2$ is clearly a principal bundle over $M$ with the principal action of the multiplicative group $\R_{>0}$ of positive reals (isomorphic to the additive group of reals \emph{via} the exponential map). But the fibers of $P'$ are diffeomorphic to $\R$, so contractible, and the corresponding topological result states that in this case the bundle admits a global section $\zs:M\to P/\Z_2$. Locally,  $\zs\,\big|_{U_\za}$ can be identified with a set of two (non-ordered pair of)  local sections $\zs_\za,\zs'_\za$ of $P$ which differ by sign. This family of local sections defines a family of local trivializations of $P$ such that the transition functions take only the values $\pm 1$.
Of course, the same is valid for line bundles, so that we get the following (cf. \cite{Grabowski:2006,Grabowski:2012a}).
\begin{theorem}\label{atlas}
For any line bundle $\zt_0:L\to M$ we can find an atlas of local trivializations
$U_\za\ti\R$ with the transition functions
$$\zf_{\za\zb}:(U_\za\cap U_\zb)\ti\R\to (U_\za\cap U_\zb)\ti\R\,,\quad
\zf_{\za\zb}(x,a)=(x,\pm a)\,.$$
\end{theorem}
The above theorem has important consequences and applications. For instance, it follows immediately that line bundles are classified by the cohomology group $H^1(M;\Z_2)$. Second, every manifold has a nowhere vanishing 1-density (cf. \cite{Grabowski:2006,Grabowski:2012}). Third, as the transition functions for the dual $L^*$ of a line bundle $L$ are algebraic inverses of the transitions functions for $L$, the bundles $L$ and $L^*$ are (non-canonically) isomorphic. Since
$L^*$ is the inverse of $L$ in the sense of K-theory, i.e., $L\ot_M L^*$ is a trivial line bundle $M\ti\R$ with the trivialization $\psi:L\ot_ML^*\to M\ti\R$ given by
$$\psi(a_x\ot a^*_x)=(x,\la a_x,a^*_x\ran)\,,$$
the line bundle $L\ot_ML$ is trivializable for any line bundle $\zt_0:L\to M$.

\medskip
It was observed in \cite{Grabowski:2009} that vector bundles are completely determined by their Euler (called also Liouville) vector fields (equivalently, by the multiplication by reals), so that lifting the Euler vector fields from a vector bundle $E\to M$ to the tangent $\sT E$ and the cotangent bundle $\sT^*E$ provides double vector bundle structures, where the second vector bundle structures are those of $\sT E\to\sT M$ and $\sT^*E\to E^*$.
Note that \emph{double vector bundles} are interesting geometric objects (see e.g. \cite{Konieczna:1999, Pradines:1974}) with a particular interest in geometrical mechanics; they allow for a natural generalization of the Hamiltonian and Lagrangian formalisms to the so called mechanics on Lie algebroids \cite{Grabowska:2006,Grabowska:2008}.
In terms of the multiplication by reals $h_s(v)=s\cdot v$, where $s\in\R$ (which is a particular action on $E$ of the multiplicative monoid $(\R,\cdot)$), the lifted multiplications by reals are
\cite{Bruce:2017,Grabowski:2009,Grabowski:2013}
\be\label{lifts}(\sT h)_s=\sT h_s\quad\text{and}\quad \hat h_s=(\sT^* h)_s=s\cdot(\sT h_{s^{-1}})^*\,,\ee
respectively (the latter makes sense even for $s=0$). In other words,
for $\za_{v_q}\in\sT^*_{v_q}(L^\ti)$ we have $\hat h_s(\za_{v_q})=s\cdot\za_{s\cdot v_q}$, where $v_q$ is clearly an element of the fibre $L_q$. Note that the phase lift  $\sT^*h$ is not the standard lift of a group action  to the cotangent bundle, that is true for the tangent lift. The Euler vector field of the vector bundle structure $\sT E\to\sT M$ is the \emph{complete} tangent lift of the Euler vector field $\n_E$ of $E$. Actually, we can lift to the tangent bundle $\sT M$ all tensor fields on $M$ \cite{Yano:1966,Yano:1973}, in particular symplectic forms and Poisson structures \cite{Grabowski:1995}.

All this can be generalized (cf. \cite{Bruce:2016,Bruce:2017}) to arbitrary actions of the monoid $(\R,\cdot)$ of multiplicative reals on a manifold $F$ which were studied and called \emph{homogeneity bundles} in \cite{Grabowski:2012}. Of course, the formulae (\ref{lifts}) can be restricted to $s\ne 0$ which produces \emph{tangent} and \emph{phase lifts} of $\Rt$-actions, as defined in \cite{Grabowski:2013}.
\begin{example}\label{pl}
For a trivial $\Rt$-principal bundle $P=\Rt\ti M$ (or for a local trivialization of an arbitrary $\Rt$-principal bundle $\zt:P\to M$) with coordinates $(s,x^i)$ and the $\Rt$-action $h_{s'}(s,x^i)=(s'\cdot s,x^i)$, the tangent lift $\sT h$ is the $\Rt$-action on $\sT P=\sT\Rt\ti \sT M$ which in the adapted coordinates $(s,\dot s,x^i,\dot x^j)$ looks like
$$(\sT h)_{s'}(s,\dot s,x^i,\dot x^j)=(s'\cdot s,s'\cdot\dot s,x^i,\dot x^j)\,.$$
The corresponding base manifold is the \emph{Atiyah algebroid} of $\sT P$ which can be identified with $\R\ti\sT M$.
Similarly, the phase lift $\hat h=\sT^*h$ is the $\Rt$-action on $\sT^*\Rt\ti\sT^*M$ given it the adapted coordinates by
$$\hat h_{s'}(s,p,x^i,p_j)=(s'\cdot s,p,x^i,s'\cdot p_j)\,.$$
In this case, the base manifold can be identified with the bundle $\sJ^1(M\ti\R^*)\simeq \sT^*M\ti\R^*$ of first jets of sections of the (trivial in this case) line bundle $M\ti\R^*\to M$. Both identifications of the base manifolds serve also in nontrivial cases (see the next paragraph) and are crucial for understanding of contact Hamiltonian and Lagrangian dynamics.
\end{example}
From the above example it is clear that the lifted $\R^\ti$-actions are again principal, so that we obtain the lifted $\Rt$-principal bundles. In the case of the tangent lift, $\sT P$ is an $\Rt$-principal bundle over the so called \emph{Atiyah algebroid} of $P$. In the case of the phase lift, $\sT^*P$ is an $\Rt$-principal bundle over $\sJ^1(L^*_P)$, i.e., the bundle of first jets of sections of the line bundle $L^*_P\to M$ (cf. \cite{Grabowski:2013}). We will study closer the latter case in the next section. The tangent lift is related to the contact Lagrangian formalism and will not be discussed in this paper.\newline

If $F:P\to\R$ is a function on an $\Rt$-principal bundle $\zt:P\to M$ with the $\Rt$-action $h$, and $k\in\Z$, then we say that $F$ is \emph{$k$-homogeneous} (or \emph{homogeneous of degree (weight) $k$}) if $F\circ h_s=s^k\cdot F$. A similar definition serves for $r$-forms $\zw$ and vector fields $X$ on $P$:
$$h^*_s(\zw)=s^k\cdot\zw\quad\text{and}\quad (h_s)_*(x)=s^k\cdot X\,.$$
Denote with $\n_P$ the vector field on $P$ which is the generator of the one-parameter group of diffeomorphisms $t\mapsto h_{e^t}$. If $P\,\big|_{U_\za}\simeq \R^\ti\ti U_\za$ is a local trivialization of $P$, then $h_{s'}(s,x)=(s'\,s,x)$ and $\n_P=s\,\pa_s$. In other words, $\n_P$ is the opposite of the fundamental vector field of the $\R^\ti$-action. An $r$-form $\zw$ (resp., a vector field $X$) on $P$ is $k$-homogeneous if and only if $\Ll_{\n_P}\zw=k\cdot \zw$ (resp.,
$\Ll_{\n_P}X=[\n_p,X]=k\cdot X$) and $(h_{-1})^*\zw=(-1)^k\,\zw$ (resp., $(h_{-1})_*X=(-1)^k\,X$).

The fundamental concept we will deal with in this paper is the following (cf. \cite{Bruce:2017,Grabowski:2013}).
\begin{definition}
A \emph{symplectic $\Rt$-principal bundle} is an $\Rt$-principal bundle $\zt:P\to M$ with the $\Rt$-action $h:\Rt\ti P\to P$, equipped additionally with a 1-homogeneous symplectic form $\zw$ on $P$. The whole structure we will denote $(P,\zt,M,h,\zw)$.
\end{definition}
\begin{example}\label{ex1}
Let $\zt_0:L\to Q$ be a line bundle. Then $P=\sT^*(\Lt)$ is a symplectic $\Rt$-principal bundle with the $\Rt$ action being the phase lift of the multiplication by non-zero reals on $\Lt$ and the canonical symplectic form $\zw_\Lt$ on $\sT^*\Lt$. Indeed, it suffices to check it for trivial line bundles $L=\R\ti Q$. Let $(s,q^i)$ be the corresponding local coordinates and $(s,p,q^i,p_j)$ be the adapted local coordinates on $\sT^*\Lt$. The lifted action is $s'.(s,p,q^i,p_j)=(s'\cdot s,p,q^i,s'\cdot p_j)$, so that the coordinates $p,q^i$ are homogeneous of degree 0, and $s,p_j$ are homogeneous of degree 1. Then the canonical symplectic form
$$\zw_\Lt=\xd s\we\xd p+\xd q^i\we\xd p_i$$
is clearly 1-homogeneous.
\end{example}
\begin{remark}
Note that the 1-homogeneous symplectic form $\zw$ is always exact. Indeed, as $\Ll_{\n_P}\zw=\zw$, we have $\zw=\Ll_{\n_P}\zw=\xd\, i_{\n_P}\zw$. Symplectic $\Rt$-principal bundles $(P,\zw)$ are particular \emph{symplectic Liouville manifolds} in the terminology of \cite{Libermann:1987}, and the vector field $\n_P$ is a particular \emph{Liouville vector field} there. However, the degree of homogeneity of a vector field on an $\Rt$-principal bundle is in \cite{Libermann:1987} different from ours.
\end{remark}
For any function on an $\Rt$-principal bundle $P$ (Hamiltonian) we associate the Hamiltonian vector field $X_H$ in the standard way, $i_{X_H}\zw=\xd H$.
\begin{proposition}\label{prop1}
On any $\Rt$-principal bundle $(P,\zt,M,h,\zw)$, 1-homogeneous functions are closed with respect to the Poisson bracket $\{ H,H'\}_\zw$ associated with the symplectic form $\zw$ (called sometimes the \emph{Lagrange bracket}). As 1-homogeneous functions are of the form $H=\zi_\zs$ for $\zs$ being sections of the line bundle $L^*_P$, this Poisson bracket induces a \emph{Jacobi bracket}
$\{\zs,\zs'\}^J_\zw$ on the $C^\infty(M)$-module of sections of $L^*_P$ defined by
$$\zi_{\{\zs,\zs'\}^J_\zw}=\{\zi_\zs,\zi_{\zs'}\}_\zw\,.$$
\end{proposition}
\begin{remark}
This bracket makes the line bundle $L^*_P$ into a \emph{local Lie algebra} in the sense of Kirillov \cite{Kirillov:1976} (see also \cite{Guedira:1984}) or \emph{Jacobi bundle} in the sense of Marle \cite{Marle:1991} (see also \cite{Bruce:2017}). This Jacobi bracket is sometimes called the \emph{Legendre bracket}.
For trivial $P$ it is a particular case of a Jacobi bracket on the $C^\infty(M)$-module $C^\infty(M)$ in the sense of Lichnerowicz \cite{Lichnerowicz:1978} (for graded Jacobi brackets we refer to \cite{Grabowski:2001,Grabowski:2003} and for supergeometric version to \cite{Grabowski:2013,Mehta:2013}). A standard misunderstanding present in the literature is that the Jacobi bracket on $C^\infty(M)$ is viewed as a bracket on the algebra $C^\infty(M)$ and not on $C^\infty(M)$ as a $C^\infty(M)$-module. In fact, the corresponding Jacobi structure (cf. \cite{Dazord:1991}) on $M$ in the form of a pair $(\zL,\zG)$, where $\zL$ is a bivector field and $\zG$ is a vector field on $M$ (satisfying some additional conditions), comes from taking the constant function 1 as the basic section for this module, while choosing another basic section leads to other tensors for the same Jacobi bracket. In contrast to that, for Poisson brackets the structure of an associative algebra on $C^\infty(M)$ is crucial. For a nice characterization of Poisson and Jacobi brackets we refer to \cite{Grabowski:1979a}.
\end{remark}
\noindent In the next section we will try to convince the reader that symplectic $\Rt$-principal bundles are nothing but contact structures.

\section{Contact structures}
Let us start with the following simple observation.
\begin{theorem}\label{t1} Let $\zh$ be a 1-form on a manifold $M$ of dimension $2n+1$. The following are equivalent.
\begin{description}
\item{(a)} The volume form $\zn=\zh\we(\xd\zh)^n$ is nowhere vanishing, $\zh\we(\xd\zh)^n\ne 0$.
\item{(b)} The 1-form $\zh$ is nowhere vanishing and the 2-form $\xd\zh$ is nondegenerate on the distribution $C=\ker(\zh)$.
\item{(c)} The vector bundle morphism
$$\#_\zh:\sT M\to\sT^*M\,,\quad \#_\zh(X)=\la X,\zh\ran\cdot\zh+i_X\xd\zh$$
is an isomorphism.
\item{(d)} The two form $\zw_\zh=\xd(s\cdot\zh)=\xd s\we\zh+s\cdot\xd\zh$ on $\bar M=\Rt\ti M$ is symplectic.
\item{(e)} The 1-form $\zh$ is nowhere vanishing and $[\zh]^\ti$ is a symplectic submanifold of $\sT^*M$ equipped with the canonical symplectic form $\zw_M$ on $\sT^*M$, where $[\zh]\subset\sT^*M$ is the line subbundle in $\sT^*M$ generated (spanned) by $\zh$.
\end{description}
\end{theorem}
\begin{proof}
(a)$\Rightarrow$(b)
The 1-form is clearly nowhere vanishing, so that $C=\ker{\zh}$ is a well-defined hyperplane field (distribution of codimension 1). Let $X_1,\dots,X_{2n}$ be vector fields from $C$ forming a local basis of $C$. Then $\xd\zh$ is nondegenerate on $C$ if and only if
$F=(\xd\zh)^n(X_1,\dots,X_{2n})$ is a nowhere vanishing function.
Suppose $F(x_0)=0$ and let us take a vector $v\in\sT_{x_0}M$ transversal to $C$, i.e., $i_v\zh(x_0)=a\ne 0$. It is clear that
$$\zn(x_0)\left(v,X_1(x_0),\dots,X_{2n}(x_0)\right)=a\cdot F(x_0)=0\,,$$
that contradicts (a).

\medskip\noindent (b)$\Rightarrow$(c) Suppose
\be\label{E1}\la X,\zh\ran(x_0)\cdot\zh(x_0)+(i_X\xd\zh)(x_0)=0\ee
for some $x_0\in M$ and a vector field $X$, $X(x_0)\ne 0$.
Contracting (\ref{E1}) with $X(x_0)$ we get $[\la X,\zh\ran(x_0)]^2=0$, so that $X(x_0)\in C_{x_0}$. Using (\ref{E1}) once more we get $(i_X\xd\zh)(x_0)=0$ that contradicts (b).

\medskip\noindent (c)$\Rightarrow$(d)
Let us see first that $\zh$ is nowhere vanishing. Since $\#_\zh$ is an isomorphism, in the case $\zh(x_0)$ we would have that
$$(\xd\zh)^{\flat}:\sT M\to\sT^*M\,,\quad (\xd\zh)^{\flat}(X)=i_X\xd\zh$$
is an isomorphism at $x_0$, so $\xd\zh$ is symplectic in a neighbourhood $U$ of $x_0$. But the manifold is odd-dimensional and does not admit symplectic forms.
Suppose now that $\zw_\zh$ is degenerate, so that $i_Y\zw_\zh(s_0,x_0)=0$ for some nonzero $Y=f\,\pa_s+X(x_0)\in\sT_{s_0}\Rt\ti\sT_{x_0}M$. This means
$$f\cdot\zh(x_0)-\la X(x_0),\zh(x_0)\ran\,\xd s+s_0\cdot i_{X(x_0)}\xd\zh(x_0)=0\,,$$
which implies that $\la X(x_0),\zh(x_0)\ran=0$ and
$$f\cdot\zh(x_0)+s_0\cdot i_{X(x_0)}\xd\zh(x_0)=0\,.$$
As $s_0\ne 0$, in the case $f=0$ we would have
$$\#_\zh(X(x_0))=\la X(x_0),\zh(x_0)\ran\cdot\zh(x_0)+i_{X(x_0)}\xd\zh(x_0)=i_{X(x_0)}\xd\zh(x_0)=0\,,
$$
so $X(x_0)=0$. But $Y\ne 0$ and hence $f\ne 0$; a contradiction. Suppose therefore that $f\ne 0$ which implies
\be\label{E2}i_{X(x_0)}\xd\zh(x_0)=(\xd\zh)^\flat(X(x_0))=-f\cdot\zh(x_0)\,.\ee
The map $\#_\zh$ on the $2n$-dimensional distribution $\ker(\zh)$ coincides with $(\xd\zh)^\flat$, so that $V=(\xd\zh)^\flat(\ker(\zh))$ is a $2n$-dimensional vector subbundle in $\sT^*M$
which, according to (\ref{E2}), contains $\zh(x_0)$. But the image of $\#_\zh$ is contained in $V+[\zh]$, where $[\zh]\subset\sT^*M$ is the line subbundle generated by the 1-form $\zh$. As $V(x_0)$ contains $\zh(x_0)$, the image of $\#_\zh(x_0)$ is contained in $V(x_0)+[\zh(x_0)]=V(x_0)$
which is not the full $\sT^*_{x_0}M$, since $V(x_0)$ has dimension $2n$.

\medskip\noindent (d)$\Leftrightarrow$(e)
The both statements imply that $\zh$ is nowhere vanishing, so that $[\zh]$ is a well-defined line subbundle of $\sT^*M$. Consider the smooth map
$$I_\zh:\Rt\ti M\to\sT^*M\,,\quad I_\zh(s,x)=s\cdot\zh(x)\,.$$
It is obvious that it provides a diffeomorphic identification of $[\zh]^\ti$ with $\Rt\ti M$. The canonical symplectic form $\zw_M$ on $\sT^*M$ can be written as $\zw_\zh=-\xd \zvy_M$ (the sign is a question of convention), where $\zvy_M$ is the canonical (tautological) Liouville 1-form on $\sT^*M$. It is well known that the pullback $\za^*(\zvy_M)=\za$ for any 1-form $\za$ on $M$, viewed as the section $\za:M\to\sT^*M$. In consequence,
$$I^*_\zh(\zw_M)(s,x)=I^*_\zh(\xd\zvy_M)(s,x)=\xd(I^*_\zh(\zvy_M))(s,x)=
\xd((s\cdot\zh)^*(\zvy_M))(s,x)=\xd s\we\zh(x)+s\cdot\xd\zh(x)\,,$$
so that $[\za]^\ti$ is a symplectic submanifold of $\sT^*M$ if and only if $\zw_\zh$ is symplectic on $\Rt\ti M$.

\medskip\noindent (d)$\Rightarrow$(a)
Since $\zw_\zh$ is symplectic on the $(2n+2)$-dimensional manifold $\bar M=\Rt\ti M$, the Liouville volume form $(\zw_\zh)^{n+1}$ is nowhere vanishing. But
$$(\zw_\zh)^{n+1}=(\xd s\we\zh+s\cdot\xd\zh)^{n+1}=(n+1)s^n\cdot\xd s\we\zh\we(\xd\zh)^n\,,$$
so that also $\zh\we(\xd\zh)^n$ is nowhere vanishing.

\end{proof}
\begin{definition} A 1-form defined on a $(2n+1)$-dimensional manifold $M$ which satisfies one of the conditions (a) -- (e) (thus all of them) we call a \emph{contact form}.
\end{definition}
The fundamental result in this area is the fact that all contact forms in dimension $2n+1$ are locally equivalent.
\begin{theorem}[contact Darboux Theorem] If $\zh$ is a contact form on a manifold $M$ of dimension $(2n+1)$, then for each $x_0\in M$ there is a neighbourhood of $x_0$ with local coordinates $(z,q^i,p_j)$, $i,1=1,\dots, n$, such that
$$\zh=\xd z-p_i\,\xd q^i\,.$$
\end{theorem}
Such coordinates we call \emph{Darboux coordinates} for $\zh$.

\medskip\noindent
From Theorem \ref{t1} (but also from the Darboux Theorem) we immediately get the following properties of contact forms.
\begin{theorem} Let $\zh$ be a contact form on a manifold $M$. Then,
\begin{itemize}
\item The 1-form $\zh$ is nowhere vanishing.
\item The one form $g\cdot\zh$ is also a contact form for any nowhere vanishing function $g:M\to\R$.
\item There exist a unique vector field $\cR_\zh$ on $M$ such that $i_{\cR_\zh}\zh=1$ and $i_{\cR_\zh}\xd\zh=0$.
\item The manifold $M$ is orientable.
\item The dimension of (immersed) submanifolds $e_N:N\hookrightarrow M$ such that $e^*_N(\zh)=0$ is $\le n$.
\end{itemize}
\end{theorem}
\noindent The vector field $\cR_\zh$ we call the \emph{Reeb vector field} for $\zh$. As for the last statement, just observe that $\Rt\ti N$ is an isotropic submanifold for the symplectic form $\zw_\zh$, so it has dimension $\le n+1$. Note that the Reeb vector fields $\cR_\zh$ and $\cR_{g\cdot\zh}$ are different. In fact, they are generally even not proportional, $\cR_{g\cdot\zh}\ne f\cdot\cR_\zh$ for any function $f$.
\begin{definition}
A distribution of hyperplanes $C\subset\sT M$, where $M$ is a manifold of odd dimension $2n+1$,  we call \emph{maximally non-integrable} if locally $C=\ker(\zh)$ for a (local) contact form $\zh$. Such distributions we call also \emph{contact structures} on $M$. Submanifolds $N\subset M$ such that  $\sT N\subset C$ we call \emph{isotropic}. If $N$ is isotropic of the maximal possible dimension (i.e., $\dim(N)=n$), we call $N$ a \emph{Legendre submanifold}. If the contact form $\zh$ may be chosen global, we call the contact structure \emph{trivializable} (\emph{coorientable}); if a single global $\zh$ is chosen, the contact structure is \emph{trivial} (\emph{cooriented}).
\end{definition}
The term `maximal non-integrability' is justified by the fact that the Frobenius integrability criterion for the distribution $C=\ker(\zh)$ is $\xd\zh=0$ on $C$. The property that $\xd\zh$ is nondegenerate on $C$, valid for contact forms, is just the other extreme.
Of course, the contact form $\zh$ is determined only up to a nowhere vanishing factor, but it is ok, since we already know that any 1-form $g\cdot\zh$ with nowhere vanishing $g:M\to\R$ is also contact and they have the same kernel. Also isotropic and Legendre submanifolds for $\zh$ and $g\cdot\zh$ are the same, so the above definition makes sense. The contact forms $\zh$ and $g\cdot\zh$ we call \emph{equivalent}. In other words, a contact structure on $M$ is an open covering $\{ U_\za\}$ of $M$ with domains equipped with equivalence classes  of local contact forms  which coincide on the intersections $U_\za\cap U_\zb$.
Note that a contact structure may not be determined by a global contact form.
\begin{example}\label{ex}
Consider the projective cotangent bundle $M=\Pe(\sT^*\R^{n+1})$ which is the set of classes $[(q^i,p_j)]$ of points of
$$P=\left(\sT^*\R^{n+1}\right)^\ti=\R^{n+1}\ti\left((\R^{n+1})^*\setminus\{ 0\}\right)$$
modulo the equivalence relation
$$(q^i,p_j)\sim(q^i,\zl\cdot p_j)\,,$$
where $\zl\ne 0$ and $i,j=0,\dots,n$. The Liouville 1-form $\zvy_{\R^{n+1}}=p_i\xd q^i$ on $\sT^*\R^{n+1}$ is nowhere vanishing on $P$ and on every coordinate neighbourhood
$$U_k=\{[(q^i,p_j)]\,\big|\, p_k\ne 0\}\subset M$$ with the projective coordinates $$(q^i,p^k_0,\dots,p^k_{k-1},p^k_{k+1},\dots,p^k_{n})\,,\quad p^k_j([(q^i,p_j)])=p_j/p_k\,,$$
it induces the contact form
$$\zh_k=\xd q^k+\sum_{i\ne k}p^k_i\,\xd q^i=\xd q^k+\sum_{i\ne k}\left(\frac{p_i}{p_k}\right)\xd q^i\,.$$
Since on $U_k\cap U_l$ the change of projective coordinates is
$$\left(q^i,p^k_j\right)\mapsto\left(q^i,\frac{p^k_0}{p^k_l},\dots,\frac{1}{p^k_l},\dots,
\frac{p^k_n}{p^k_l}\right)\,,$$
we have
$$\zh_l=\xd q^l+\sum_{i\ne l,k}\left(\frac{p^k_i}{p^k_l}\right)\xd q^i+\left(\frac{1}{p^k_l}\right)\xd q^k=\left(\frac{1}{p^k_l}\right)\zh_k\,,$$
so that $\zh_l$ and $\zh_k$ are equivalent on $U_l\cap U_k$. This defines a canonical contact structure on $\Pe(\sT^*\R^{n+1})$. Of course, $\Pe(\sT^*\R^{n+1})$ is topologically $\R^{n+1}\ti\Pe\R^n$ and $\Pe\R^n$ is not orientable for even $n$, so that there is no globally defined contact form on
$\Pe(\sT^*\R^{n+1})$ if $n$ is even.
By a similar method we can define a canonical contact structure on the projective cotangent bundle $\Pe(\sT^*N)$ for any manifold $N$, which does not admit a global contact form if $\dim(N)$ is odd.
\end{example}
Note that the line subbundle $[\zh]$ generated by a contact form $\zh$ does not depend on the choice of a contact form from the equivalence class of $\zh$, so that $[\zh]$ determines an equivalence class of contact forms rather than a single contact form. In fact, $[\zh]$ is the annihilator $C^o$ of the field of hyperplanes $C$. This observation, together with Theorem \ref{t1}, implies the following.
\begin{theorem} Let $L$ be a line subbundle in $\sT^\ast{M}$. The followings are equivalent:
\begin{description}
\item{(a)}  $L$ is locally generated by contact one-forms.
\item{(b)}  $L^\ti$ is a symplectic submanifold of $\sT^\ast{M}$.
\item{(c)}  The annihilator $C=L^o\subset\sT M$ is a contact structure on $M$.
\end{description}
\end{theorem}
It is clear that $L^\ti$ is canonically an $\Rt$-principal bundle with the $\Rt$-action induced from the multiplication by reals in the vector bundle $\sT^*M$. The canonical symplectic form $\zw_M$ on $\sT^*M$ is 1-homogeneous, so its restriction to $L^\ti$ is a symplectic form which is 1-homogeneous with respect to the $\Rt$-action. In other  words, if $C\subset\sT M$ is a contact structure on $M$, then $(C^o)^\ti$ is canonically a symplectic $\Rt$-principal bundle. This observation leads us to the fundamental result of this section which was proved in a more general situation (for contact structures on supermanifolds) in \cite{Grabowski:2013}.
\begin{theorem}
There is a canonical one-to-one correspondence between contact structures $C\subset\sT M$ on a manifold $M$  and symplectic $\Rt$-principal bundles over $M$. In this correspondence,  the canonical symplectic $\Rt$-principal bundle associated with $C$ is $(C^o)^\ti\subset\sT^*M$ (and will be called the \emph{symplectic cover} of the contact manifold $(M,C)$).
\end{theorem}
\begin{proof}
We need only to prove that any symplectic $\Rt$-principal bundle $(P,\zt,M,h,\zw)$ is canonically isomorphic with the symplectic $\Rt$-principal bundle structure on a symplectic submanifold $L^\ti$ in $(\sT^*M,\zw_M)$, where  $L\subset\sT^*M$ is a line subbundle.

If $\zD$ is the Euler vector field associated with $h$, then the 1-form $\wt{\za}=i_{\zD}\zw$ on $P$ is
semi-basic and defines a map $\Psi:P\to\sT^\ast{M}$ which is an isomorphism of the principal $\R^\ti$-bundle
$P$ onto $C^\ti\subset\sT^\ast{M}$, where $C$ is the line subbundle in $\sT^\ast{M}$ spanned by the image of
$\wt{\za}$. Indeed, let us write in local bundle coordinates $(t,x)$ in $P$, with $\zD=t\pa_t$, the symplectic
form $\zw$ as $\zw=\xd t\we{\za}+\zw'$ for a certain semi-basic one-form ${\za}$ and a semi-basic two-form
$\zw'$. Since $i_\zD\zw'=0$ and  $i_{\zD}\za=0$, we get $\wt{\za}=t\cdot\za$. As $\zD$ is 0-homogeneous,
$i_{\zD}\zw$ is 1-homogeneous, so $\za$ is homogeneous of degree 0, thus $\R^\ti$-invariant. Being
simultaneously semi-basic it is actually basic, so it can be regarded as a 1-form on ${M}$. Moreover,
$\xd(t\cdot\za)=\zw$, so that $\za$ is a contact form. Hence, $\Psi(t,x)=t\cdot\za(x)$ is a principal bundle
isomorphism which does not depend on the choice of homogeneous coordinates $(t,x)$, since a change in the
choice of the local trivialization of $P$ results in multiplication of $\za$ by an invertible function on
${M}$ and in multiplication of $t$ by the inverse of this function.
\end{proof}
\begin{proposition}
Let $(P,\zt,M,h,\zw)$ be a  contact structure on $M$.  An (immersed) submanifold $e_N:N\hookrightarrow M$  is a Legendre  submanifold of $M$ if and  only  if its pullback $\zt^{-1}(N)$ to $P$ is a Lagrangian submanifold in $(P,\zw)$. In other words, there is a canonical one-to-one correspondence between $\Rt$-invariant Lagrangian submanifolds $\cL$ of $P$ and Legendre  submanifolds $\cL_0=\zt(\cL)$ of $M$.
\end{proposition}
\begin{proof} It is obvious that a  submanifold $\cL$ of $P$ is invariant with respect to the action of $\Rt$ if and only if it is a pullback (inverse image) of a submanifold $\cL_0$ of $M$, $\cL=\zt^{-1}(\cL_0)$. In a local trivialization $P=\Rt\ti M$ we can write
$$\cL=\Rt\ti\cL_0=\{(s,x)\in\Rt\ti M\,\big|\,x\in\cL_0\}\,,$$
and the symplectic form $\zw$ reads $\zw=\xd s\we\zh+s\cdot\xd\zh$ for some (local) contact form $\zh$ on $M$. Local vector fields $X$ tangent to $\cL$ are of the form $X_{(s,x)}=A(s,x)\,\pa_s+Y_s(x)$, where $A(s,x)\in\R$ and $Y_s(x)$ is tangent to $\cL_0$ for $x\in\cL_0$. It is easy to see now that
$$\zw(X,X')(s,x)=(A(s,x)i_{Y'_s}\zh)(x)-(A'(s,x)i_{Y_s}\zh)(x)+s\cdot\xd\zh(Y_s,Y'_s(x))(x)=0$$ for all $(s,x)\in\cL$ and all $X,X'$ tangent to $\cL$ if and only if $Y_s(x),Y'_s(x)\in\ker(\zh)$, i.e., $\sT\cL_0\subset\ker(\zh)$.

\end{proof}
\begin{definition} The symplectic $\Rt$-principal bundle associated with a contact manifold $(M,C)$ we will call the \emph{symplectic cover of $(M,C)$}.
\end{definition}
\begin{remark}
Many elements of the above  approach to contact structures are already present in books by Arnold \cite{Arnold:1989} and  Libermann--Marle \cite{Libermann:1987} and  spread over the literature. They are associated with the concept of a \emph{symplectization} of a contact structure. Actually, $(C^o)^\ti$ is the symplectization of the contact structure $C$. However, it is more sophisticated than the standard symplectization $\xd(e^s\,\zh)$ of a contact form $\zh$, which lives on $\R\ti M$. We need nontrivial $\Rt$-principal bundles with fibers $\Rt$ rather than $\R$ to cover the cases of nontrivial contact structures. In this paper we prefer, finding it much more elegant
and fruitful, to view the symplectizations as the contact structures themselves and to identify contact structures with their symplectic covers.
\end{remark}
\begin{example}\label{jb} Let $\zt_0:L\to Q$ be a line bundle and $\zt^*_0:L^*\to Q$ be its dual. As we already know (Example \ref{ex1}) the symplectic manifold $P=\sT^*(L^\ti)$ is canonically a symplectic $\R^\ti$-principal bundle.
It is also well known (cf. \cite{Grabowski:2013}) that there is a canonical identification of the contact manifold $M=\sT^*(L^\ti)/\Rt$ with the bundle ${\zt^1(L^*)}:\sJ^1(L^*)\to Q$ of first jets of sections of the line bundle $\zp_0:L^*\to Q$. Indeed, consider a function $F:L^\ti\to\R$ such that $\za_{v_q}=\xd F(v_q)$. Note that $F$ can be chosen 1-homogeneous, i.e., of the form $F=\zi_S\,\big|_{L^\ti}$ for some section $S:Q\to L^*$ of the line bundle $L^*$. Indeed, in local coordinates $(s,q^i)$ associated with a local trivialization $\R^\ti\ti Q$ of $\Lt$ (thus $L$), the dual local coordinates $(z,q^i)$ for the dual bundle $L^*$, and for an arbitrary $F$, we have
$$\xd F(s_0,q_0)=\frac{\pa F}{\pa s}(s_0,q_0)\,\xd s+\sum_i\frac{\pa F}{\pa q^i}(s_0,q_0)\,\xd q^i\,.$$
Consider a local section $z=S(q)$ of $L^*=\R^*\ti Q$ represented by a function $S(q)$ on $Q$. Then
$\zi_S(s,q)=s\cdot S(q)$ and
$$(\xd\,\zi_S)(s_0,q_0)=S(q_0)\,\xd s+ s_0\cdot\sum_i\frac{\pa S}{\pa q^i}(q_0)\,\xd q^i\,,$$
so it suffices to take the section $S$ such that
\be\label{jb1}S(q_0)=\frac{\pa F}{\pa s}(s_0,q_0)\quad\text{and}\quad
\frac{\pa S}{\pa q^i}(q_0)=s_0^{-1}\cdot\frac{\pa F}{\pa q^i}(s_0,q_0)\,.\ee
Of course, there are many such sections $S$, but equations (\ref{jb1}) show that all of them have the same first jet at $q_0\in Q$. This defines a canonical projection $\zt:\sT^*(L^\ti)\to \sJ^1(L^*)$, which in the adapted coordinates $(s,q^i,p,p_j)$ in $\sT^*\Lt$ and the adapted coordinates $(z,q^i,\p_j)$ in $\sJ^1(L^*)$ (so that $j^1(S)(q)=(S(q),q_i,\frac{\pa S}{\pa q^j}(q)$) reads $\zt(s,q^i,p,p_j)=(p,q^i,p_j/s)$ and has fibers being orbits of the $\R^\ti$-action. Hence, $M=\sJ^1(L^*)=\sT^*(L^\ti)/\Rt$, and $\zt$ is exactly the projection $P\to P/\Rt=M$.
It  is easy to see that the $\Rt$-action and  the multiplication by reals  in the  vector bundle $\sT^*(\Lt)$ commute, so that we get a \emph{double principal-vector bundle structure} on  $\sT^*(\Lt)$ (cf. \cite{Grabowski:2013}),
$$\xymatrix{
\sT^*(L^\ti)\ar[rr]^{\zt} \ar[d]^{\zp_{\Lt}} && \sJ^1(L^*)\ar[d]^{\zt^1(L^*)} \\
{L^\ti}\ar[rr]^{{\zt_0}} && Q\,. }
$$
This is a particular example of a \emph{weighted principal bundle} (more precisely, a \emph{$\VB$-principal bundle}) defined in \cite{Grabowska:2022}.

\medskip\noindent The symplectic $\Rt$-principal bundle $P=\sT^*(L^\ti)$ we  call the \emph{canonical contact structure} on $\sJ^1(L^*)$. For the chosen local trivialization of $L$ (thus $L^*$), the corresponding local contact 1-form $\zh$ on $\sJ^1(L^*)$ has the Darboux form
$$\zh=\xd z-\p_i\,\xd q^i\,.$$
The canonical line bundle $L_P$ associated with $P$ is the base manifold of the $\Rt$-principal bundle $P\ti\R$ with the action $s.(x,u)=(\hat h_{s}(x),s^{-1}\cdot u)$, where $x\in P$ and $u\in\R$, i.e.,
$$s.(\za_{v_q},u)=(s\cdot\za_{s\cdot v_q},s^{-1}\cdot u)\,.$$
Considering separately the action $\hat h$ on $P$ and the action $s.(v_q,u)=(s\cdot v_q,s^{-1}\cdot u)$ on $L^\ti\ti\R$, we get $L_P$ in the form $L_P=\sJ^1(L^*)\ti_Q L$ with the identification
$$L_P\ni[(\za_{v_q},u)]\mapsto([\za_{v_q}],[(v_q,u)])\in\sJ^1(L^*)\ti_QL\,,$$
where $[\cdot]$ denotes the corresponding classes of the $\R^\ti$-actions.
Hence,
\be\label{fjs} L^*_P=\sJ^1(L^*)\ti_QL^*\,,\ee
with the obvious projections onto $\sJ^1(L^*)$ as the first factor. In other words, we have the canonical identification
$$ P=\sT^*(L^\ti)=\sJ^1(L^*)\ti_QL^\ti\,.$$
Note that the images $\sj^1(S)(Q)$ of first jets of sections $S$ of the line bundle $L^*\to Q$ are Legendre submanifolds in  $\sJ^1(L^*)$. Indeed, for the canonical contact structure $\zt:\sT^*(L^\ti)\to\sJ^1(L^*)$ on $\sJ^1(L^*)$ we have
$$\zt^{-1}\left(\sj^1(S)(Q)\right)=(\xd\,\zi_S)(L^\ti)\,,$$
and the latter submanifold of $\sT^*(L^\ti)$ is clearly Lagrangian.

\medskip\noindent
Let us consider now the trivial line bundle $L$, so the canonical contact manifold $M=\sJ^1(Q,\R)=\R\ti\sT^*Q$ with the contact form $\zh(z,q^i,\p_j)=\xd z-\p_i\xd q^i$. The contact structure is represented by the trivial $\R^\ti$-principal bundle $P=\R^\ti\ti M$ with the homogeneous symplectic form $\zw_\zh=\xd s\we\zh+s\,\xd\zh$. According to \cite{Grabowski:2013}, we can also write
$$P=\sT^*(\R^\ti\ti Q)=\R^\ti\ti\R\ti\sT^*Q$$ which is equipped with the canonical symplectic form $\zw=\zw_{\R^\ti\!\ti Q}$. In local coordinates $(s,q^i,z,p_j)$ on $\sT^*(\R^\ti\ti Q)$ and coordinates $(z,q^i,\p_j)$ on $\sJ^1(Q,\R)$ the symplectic form $\zw$ reads
$$\zw=\xd s\we\xd z+\xd q^i\we\xd p_i\,,$$
while the $\R^\ti$-action $h$ on $P$ is the cotangent lift of the canonical $\R^\ti$-action on $\R^\ti\ti Q$ and takes the form
$$h_{s_0}(s,q^i,z,p_j)=(s_0\cdot s,q_i,z,s_0\cdot p_j)\,.$$
The projection $\zt:P\to M$ in these coordinates looks like
$$\zt(s,q^i,z,p_j)=(z,q^i,\p_j=p_j/s)\,.$$
It is easy to see that the canonical symplectic form $\zw_{\R^\ti\ti Q}$ coincides with $\zw_\zh$.
Writing $\zw_{\R^\ti\ti Q}$ in coordinates $(s,q^i,z,\p_j)$, we get
$$\zw_{\R^\ti\!\ti Q}=\xd s\we\xd z+\xd q^i\we\xd(s\,\p_i)=\xd s\we(\xd z-\p_i\xd q^i)-s\,\xd \p_i\we\xd q^i\,.$$
\end{example}

\section{Hamiltonian dynamics on contact manifolds}
The commonly accepted in the literature approach to contact Hamiltonian dynamics is constructed almost exclusively only for trivial contact manifolds, i.e. manifolds $M$ equipped with a globally defined contact 1-form $\zh$.

\medskip
For a real valued function ${\hat H}$ (\emph{contact Hamiltonian}) on a contact manifold $({M},\eta)$, the corresponding \emph{contact Hamiltonian vector field} $X^c_{\hat H}$ is a vector field on $M$ defined as the unique one satisfying
\begin{equation}
i_{X^c_{{\hat H}}}\eta =-{\hat H},\qquad i_{X^c_{{\hat H}}}\xd\eta =\xd {\hat H}-\mathcal{R}({\hat H}) \eta\,,   \label{contact}
\end{equation}%
where $\mathcal{R}$ is the Reeb vector field for $\zh$. In this sense, a \emph{contact Hamiltonian system} is the triple  $({M},\eta,{\hat H})$. Since
\begin{equation}\label{L-X-eta}
\mathcal{L}_{X^c_{{\hat H}}}\eta =
\xd\,i_{X^c_{{\hat H}}}\eta+i_{X^c_{{\hat H}}}\xd\eta= -\mathcal{R}({\hat H})\eta\,,
\end{equation}
$X^c_{\hat H}$ is a contact vector field on $M$ with the conformal factor $\lambda=\mathcal{R}({\hat H})$.
In this realization, the contact Jacobi bracket of two smooth functions on ${M}$ is defined by
\begin{equation}\label{cont-bracket}
\{\hat  F,\hat H\}_\zh=i_{[X^c_{\hat F},X^c_{\hat H}]}\eta\,.
\end{equation}

According to \eqref{L-X-eta}, the flow of a contact Hamiltonian vector field preserves the contact structure, but it does not preserve neither the contact one-form nor the Hamiltonian function. Instead we obtain
$${\mathcal{L}}_{X^c_{\hat H}} \, {\hat H} = - \mathcal{R}({\hat H}) {\hat H}\,.$$
Referring to Darboux coordinates $(z,q^i,p_j)$, the {H}amiltonian vector field determined in \eqref{contact}, is computed to be
\begin{equation}\label{con-dyn}
X^c_{\hat H}=\frac{\partial {\hat H}}{\partial p_i}{\partial_{q^i}}  - \left(\frac{\partial {\hat H}}{\partial q^i} + \frac{\partial {\hat H}}{\partial z} p_i \right)
{\partial_{p_i}} + \left(p_i\frac{\partial {\hat H}}{\partial p_i} - {\hat H}\right){\partial_z},
\end{equation}
whereas the contact Jacobi bracket \eqref{cont-bracket} is
\begin{equation}\label{Lag-Bra}
\{\hat F,{\hat H}\}_\zh = \frac{\partial \hat F}{\partial q^i}\frac{\partial {\hat H}}{\partial p_i} -
\frac{\partial \hat F}{\partial p_i}\frac{\partial {\hat H}}{\partial q^i} + \left(\hat F  - p_i\frac{\partial \hat F}{\partial p_i} \right)\frac{\partial {\hat H}}{\partial z} -
\left({\hat H}  - p_i\frac{\partial {\hat H}}{\partial p_i}\right)\frac{\partial \hat F}{\partial z}.
\end{equation}
So, the Hamilton's equations for ${\hat H}$ read
\begin{equation}\label{conham}
\dot{q}^i= \frac{\partial {\hat H}}{\partial p_i}, \qquad \dot{p}_i = -\frac{\partial {\hat H}}{\partial q^i}-
p_i\frac{\partial {\hat H}}{\partial z}, \quad \dot{z} = p_i\frac{\partial {\hat H}}{\partial p_i} - {\hat H}.
\end{equation}
Of course, for general contact structures on $M$ equations (\ref{contact}) do not make any sense, since the Reeb vector field is not defined; the Reeb vector fields associated with different local contact 1-forms generating the contact structure are different, so do not produce any geometrical object. This also the reason why the vector field, called \emph{evolution vector field}, defined in the literature by $\mathcal{E}_{\hat H} = X_ {\hat H} + {\hat H}\cR$, has generally no geometrical meaning.
On the other hand, the Reeb vector field is not needed at all; the contact dynamics is well (and much simpler) defined if we assume that Hamiltonian functions for a contact structure $(M,C)$ live actually on its symplectic cover $(P,\zt,M,\zw,h)$, and not on $M$.
\begin{definition}
Let $(P,\zt,M,\zw,h)$ be a symplectic cover of a contact manifold $(M,C)$. A \emph{contact Hamiltonian} for this contact structure is a 1-homogeneous function $H:P\to\R$, i.e., $H\circ h_s=s\cdot H$.
\end{definition}
\begin{theorem}
Let $(P,\zt,M,\zw,h)$ be the symplectic cover of a contact manifold $(M,C)$ and let $H:P\to\R$ be a contact Hamiltonian. Then  the Hamiltonian vector field $X_H$ on $P$ associated with the symplectic form $\zw$ is $\zt$-projectable.
\end{theorem}
\begin{proof}
Note first that the Hamiltonian vector field $X_H$ is of weight 0, i.e. it is invariant with respect to the $\R^\ti$-principal action, $(h_s)_*(X_H)=X_H$. Indeed, as $\zw$ is 1-homogeneous
and $\xd H$ is also 1-homogeneous, the vector field $X_H$ is homogeneous of weight 0, since $\zw$ is non-degenerate. Explicitly,
$$s^{-1}\cdot i_{(h_s)_*(X_H)}\zw=i_{(h_s)_*(X_H)}(h_{s^{-1}})^*(\zw)=(h_{s^{-1}})^*(i_{X_H}\zw)
=(h_{s^{-1}})^*(\xd H)=s^{-1}\cdot\xd H\,,$$
so that $(h_s)_*(X_H)=X_H$. In other words, the Hamiltonian vector field is invariant with respect to the $\R^\ti$-principal action on $P$. It is well-known that such vector fields on principal bundles are projectable. In fact, they can be identified with sections of the corresponding Atiyah algebroid, and the projections represent the anchors of these sections.
\end{proof}
\begin{definition}
In the notation of the above theorem, we call the vector field $X^c_H=\zt_*(X_H)$ on $M=P/\R^\ti$ the \emph{contact Hamiltonian vector field} associated with the Hamiltonian $H$.
\end{definition}
\begin{remark}
According to Proposition \ref{prop1}, the 1-homogeneous Hamiltonians $H$ on $P$ can be identified with sections $\zs$ of the dual line bundle $\zt^*_P:L^*_P\to M$ as Hamiltonians $H_\zs=\zi_\zs$. We will call them \emph{Hamiltonian sections of the contact structure}. The contact Hamiltonian vector field $X_{\zi_\zs}$ we will denote simply $X_\zs$, and the corresponding contact Hamiltonian vector field with $X^c_\zs$. The homogeneous symplectic form $\zw$ induces the \emph{contact Jacobi bracket} $\{\zs,\zs'\}_\zw^J$ of sections of $L^*_P$.
We get the contact Jacobi bracket (\ref{Lag-Bra}) for trivial contact structures expressed in the corresponding Darboux coordinates.
Another way of defining the contact dynamics directly from a section $\zs$ is based on the observation that to a Hamiltonian section $\zs$ of $\zt^*_P:L^*_P\to M$ there corresponds a linear first-order differential operator
$$D_\zs=\ad_{\zs}:L^*_P\to L^*_P\,,\quad  D_{\zs}(\zs')=\{\zs,\zs'\}^J_\zw\,.$$
The contact Hamiltonian vector field $X^c_\zs$ on $M$ is represented simply by the principal symbol of $D_{\zs}$, i.e.,
$$D_\zs(f\zs')-f\,D_\zs(\zs')=X^c_\zs(f)\,\zs'\,.$$
\end{remark}
\begin{example}
Let $P=\R^\ti\ti M$ be the trivial $\R^\ti$-principal bundle with coordinates $(s,y=(y^a))$. Any trivial contact structure on $P$ consists of the symplectic form $\zw_\zh$ associated with a contact 1-form $\zh$ on $M$ and defined by $\zw_\zh=\xd s\we\zh+s\,\xd\zh$. If $\cR$ is the Reeb vector field for $\zh$, then $\cR$ viewed as homogeneous vector field of weight 0 on $P$, $\cR(s,y)=\cR(y)$ is the Hamiltonian vector field with respect to $\zw_\zh$ with the Hamiltonian $-s$, $i_\cR\zw_\zh=-\xd s$. Let us take a function $\hat H$ on $M$ and consider the homogeneous Hamiltonian
$$H:P\to\R\,,\quad H(s,y)=s\cdot\hat H(y)\,.$$
The function $\hat H$ we will call a \emph{reduced contact Hamiltonian}. The corresponding Hamiltonian vector field $X_H$ is homogeneous of weight 0, so that
$X_H(s,y)=F(y)\,s\cdot\pa_s+Y(y)$, where $F$ is a (pull-back of) function on $M$ and $i_Y\xd s=0$, so that the vector field $Y$ can be viewed as tangent to $M$ and therefore identified with the contact Hamiltonian vector field $X^c_H=(\zt)_*X_H$. In other words,
\be\label{cHvf} X_H(s,y)=s\,F(y)\cdot\pa_s+X^c_H(y)\,.\ee
We have
$$(sF)\cdot\zh-(i_{X^c_H}\zh)\,\xd s+s\cdot i_{X^c_H}\xd\zh=\hat H\,\xd s+s\,\xd\hat H\,,$$
so that $i_{X^c_H}\zh=-\hat H$ and
$$i_{X^c_H}\xd\zh=\xd\hat H-F\,\zh\,.$$
Contracting both sides with $\cR$, we get $F=\cR(\hat H)$, and in this way we reconstructed the equations (\ref{contact}) for the reduced contact Hamiltonian $\hat H$, and in a Darboux coordinates $y=(z,q^i,p_j)$ for $\zh$ we recover the contact Hamilton equations (\ref{conham}).
\end{example}
\begin{remark}
The above example shows that, despite of the fact that (\ref{conham}) are formulated in terms of objects not having a geometrical meaning for a general contact structure on $M$, they define properly the contact Hamiltonian dynamics if only we understand it as induced from a 1-homogeneous Hamiltonian $H$ on the symplectic $\R^\ti$-principal bundle $\zt:P\to M$, equipped with a 1-homogenous symplectic form $\zw$. Any local trivialization
$P=\R^\ti\ti M$ with elements $(s,x)$, where $s\in\R^\ti$ and $x\in M$, defines the local reduced contact Hamiltonian $\hat H=s^{-1}\,H$, the local contact 1-form $\zh=s^{-1}i_{\n_P}\zw$, and the local Reeb vector field understood as the semi-basic and projectable vector field $\cR$ on $P$ being the Hamiltonian vector field for the Hamiltonian $-s$, i.e., $i_\cR\zw=-\xd s$. The equations (\ref{conham})define in these terms the local form of the contact Hamiltonian vector field $X^c_H$. However, equations (\ref{conham}) look much more complicated than the global expression $X^c_{\hat H}=\zt_*(X_H)$ but nevertheless is correct, even that $\hat H$ has generally no geometrical meaning. Note that, according to Theorem \ref{atlas}, we can always choose local reduced Hamiltonians $\hat H_\za$ such that on intersections of charts $U_\za$ the differ just by sign.

The situation is much different if the evolution vector field is concerned. This vector field  differs from $X^c_{{\hat H}}$ by a vector field whose local expression in our local trivialization is ${\hat H}\cdot\cR$. This vector field has a geometrical meaning only for trivial contact structures, since changing the local trivialization into $(s',x)$, where $s'=g(x)\cdot s$ for a nowhere-vanishing function $g$ on $M$, we get a new local reduced Hamiltonian on $M$ in the form ${\hat H}'(x)=g^{-1}(x)\cdot {\hat H}(x)$, and a new local Reeb vector field $\cR'$. The point is that in general ${\hat H}'\,\cR'\ne {\hat H}\,\cR$.
Indeed, if ${\hat H}'\,\cR'= {\hat H}\,\cR$, then $\cR=g^{-1}\,\cR'$, that is generally not true, since
$$-\xd s=i_\cR\zw=g^{-1}\,i_{\cR'}\zw=-g^{-1}\,\xd s'=-g^{-1}\xd(gs)=-\xd s-s\,\xd(\ln|g|)$$
and $\ln|g|$ is constant only for constant $g$. This is the reason why in this paper we pay no attention to `evolution vector fields'.
\end{remark}
\begin{example} For a manifold $M$ consider the canonical contact structure on the projectivization of the cotangent bundle $\Pe(\sT^*M)$ (see Example \ref{ex}). The symplectic cover is $(\sT^*M)^\ti$ and contact Hamiltonians are functions $H:\sT^*M\setminus\{ 0_M\}\to\R$ such that, in Darboux coordinates $(q^i,p_j)$ on $\sT^*M$,
$$H(q,s\cdot p)=s\cdot H(q,p)\quad\text{for all}\quad s\ne 0\,.$$
In particular, functions linear in momenta, $H(q,p)=f^i(q)\,p_i$, i.e., of the form $\zi_X$, where $X$ is a vector field on $M$, are contact Hamiltonians. Our contact structure, thus the principal bundle $\zt:(\sT^*M)^\ti\to\Pe(\sT^*M)$, is non-trivial for $M$ of an odd dimension. In this case, such Hamiltonians do not allow for reduced contact Hamiltonians $\hat H$ defined on the projectivization $\Pe(\sT^*M)$ of the cotangent bundle, since 1-homogeneous functions on $\zt:(\sT^*M)^\ti$ project through $\zt$ if only if they are invariant with respect to the antipodal map $p\mapsto -p$. This means that our contact Hamiltonian vector fields $X^c_H$ cannot be obtained by the standard methods present in the literature.

A simple numerical example is the following. Put $M=\R^3$, and let $p=(p_1,p_2,p_3)$ be the canonical momenta on $(\sT^*\R^3)$, thus $(\sT^*\R^3)^\ti$, associated with coordinates $q=(q^1,q^2,q^3)$ on $\R^3$. The Hamiltonian $H=q^i\,p_i$ is 1-homogeneous and produces the Hamiltonian vector field $$X_H=\sum_i\left(q^i\pa_{q^i}+p_i\pa_{p_i}\right)$$ on $\sT^*\R^3$, thus on $(\sT^*\R^3)^\ti$. This Hamiltonian vector field projects \emph{via} the canonical projection $\zt:(\sT^*\R^3)^\ti\to\Pe(\sT^*\R^3)$ onto a vector field on $\Pe(\sT^*\R^3)$. This is exactly the contact Hamiltonian vector field $X^c_H$ associated with our 1-homogeneous Hamiltonian. But $\Pe(\sT^*\R^3)=\R^3\ti\Pe\R^2$ is not orientable, so the contact structure there is nontrivial, and our contact Hamiltonian vector field cannot be obtained as the pull-back of a `contact Hamiltonian` understood as a function on the contact manifold $\Pe(\sT^*\R^3)$.

Note finally that 1-homogeneous Hamiltonians on $(\sT^*\R^3)^\ti$ are uniquely determined by  functions on $\R^3\ti S^2$ satisfying $f(q,A(\zg))=\pm f(q,\zg)$, where $A:S^2\to S^2$ is the antipodal map of the sphere $S^2$. They are pull-backs of functions on $\Pe(\sT^*\R^3)$ if and only if $f(q,A(\zg))=f(q,\zg)$.
\end{example}
\section{Time-dependent (nonautonomous) Hamiltonians}
Let $(P,\zw)$ be a symplectic manifold (in what follows one can take also a Poisson tensor instead of the symplectic form $\zw$). A \emph{time-dependent (nonautonomous) Hamiltonian} for $(P,\zw)$ is a function $H:P\ti\R\to\R$. In fact, one can take $H$ to be defined only on a neighbourhood of the submanifold $P\ti\{ 0\}$ in $P\ti\R$, but for simplicity we will work on the whole $P\ti\R$. The corresponding non-autonomous dynamics is represented by the time-dependent vector field $X_t$ on $P$, where $X_t=X_{H_t}$ is the Hamiltonian vector field of $H_t(x)=H(x,t)$ with respect to $\zw$. The trajectory $\zg:(a,b)\to P$ of $X_t$ with the initial condition $\zg(t_0)=x_0$, where $t_0\in(a,b)$, is the unique solution of the differential equation
$$\dot\zg(t)=X_t(\zg(t))\,,\quad \zg(t_0)=x_0\,.$$
A convenient way of thinking about the non-autonomous dynamics $X_t$ is \emph{via} its \emph{autonomization}, i.e. to consider the vector field $\bar X$ on $P\ti\R$, $\bar X(x,t)=X_t(x)+\pa_t$. Then  we consider the autonomous dynamics on $P\ti\R$ induced by $\bar X$. The projection of the trajectory of the vector field $\bar X$ with the initial point $(x_0,t_0)$ to $P$ is exactly the trajectory of the time-dependent vector field $X_t$ with the same initial condition.

To obtain the autonomization of the time-dependent Hamiltonian vector field $X_{H_t}$ we can proceed as follows. On the manifold $\bar P=P\ti\sT^*\R$ consider the symplectic form $\bar\zw=\zw+\zw_\R$, where $\zw_\R$ is the canonical symplectic form on the cotangent bundle $\zp_\R:\sT^*\R\to\R$. Using the canonical global Darboux coordinates $(t,p)$ on $\sT^*\R$ we have
$$\bar\zw(x,t,p)=\zw(x)+\xd t\we\xd p\,.$$
On $\bar P$ consider the Hamiltonian $\bar H(x,t,p)=H(x,t)+p$, and the corresponding Hamiltonian vector field
\be\label{aH}X_{\bar H}(x,t,p)=X_{H_t}(x)-\frac{\pa H}{\pa t}(x,t)\pa_p+\pa_t\,.\ee
If we denote the canonical projection $(\id,\zp_\R):P\ti\sT^*\R\to P\ti\R$ with $\bar\zp$, and the canonical projection $P\ti\sT^*\R\to P$ with $\zp$, then the autonomization of $X_{H_t}$ is $(\bar\zp)_*(X_{\bar H})$ and the trajectories of the time-dependent vector field $X_t=X_{H_t}$ are exactly the projections $\zg(t)=(\zp\circ\bar\zg)(t)$ of trajectories $\bar\zg(t)$ of the Hamiltonian vector field $X_{\bar H}$. In other words, the nonautonomous dynamics associated with a time-dependent Hamiltonian $H(x,t)$ on $P\ti\R$ is given by the time-dependent Hamiltonian vector field $X_{H_t}$.

\subsection{The contact case}
To study nonautonomous contact Hamiltonian dynamics, we will work with autonomizations of contact structures.
\begin{definition}
For a contact manifold $M$ with the contact structure $(P,\zt,M,\zw,h)$ the \emph{autonomization of this contact structure} is the contact structure on the manifold $\bar M=M\ti\sT^*\R$ of the form
$(\bar P,\bar\zt,\bar M,\bar\zw,\bar h)$, where $\bar P=P\ti\sT^*\R=P\ti\R\ti\R^*$ is an $\R^\ti$-principal bundle with the $\R^\ti$-action being the product of $h$ and the standard action on a vector bundle by multiplication of vectors by reals, $\bar h_s(x,t,p)=(h_s(x),t,s\cdot p)$, so that $\bar\zt$ is the canonical projection
$$\bar P=P\ti\sT^*\R\raa\bar P/\R^\ti\simeq \bar M\,,$$
and the symplectic form reads $\bar\zw=\zw+\xd t\we\xd p$ (it is clearly 1-homogeneous).
\end{definition}
\noindent Since $(P\ti\R)/\R^\ti$ for the $\Rt$-action $\bar h_s(x,p)=(h_s(x),s\cdot p)$ is $L^*_P$, we can identify $\bar M$ with $\bar L^*_P=L^*_P\ti\R$ which is canonically a line bundle over $M\ti\R$,
$$\bar\zt^*_P:\bar L^*_P=L^*_P\ti\R\to M\ti\R\,.$$
Hence, the projection $\bar\zt$ acts as
$$\bar\zt:\bar P=P\ti\sT^*\R\to L^*_P\ti\R=\bar L^*_P=M\ti\sT^*\R\,.$$

\medskip
Let now $H:P\ti\R\to\R$ be a time-dependent contact Hamiltonian on $P$, i.e., $H_t=H(\cdot,t)$ is a contact (1-homogeneous) Hamiltonian on $P$ for all $t\in\R$. As above, we define the autonomization of $H$ as the Hamiltonian $\bar H:\bar P=P\ti\sT^*\R\to \R$ of the form $\bar H(x,t,p)=H(x,t)+p$; this Hamiltonian is a contact (1-homogeneous) Hamiltonian on $\bar P$.
We already know that the Hamiltonian vector field $X_{\bar H}$ reads as in (\ref{aH}), and it projects onto the contact Hamiltonian vector field $X^c_{\bar H}$ on $\bar L^*_P$.
Since the projection $\bar\zt$ reads
\be\label{barzt}\bar\zt(x,t,p)=(p\cdot i^*_P(x),t)\,,\ee
where
$$i^*_P:P\to L^*_P\,,\quad i^*_P(x)=[(x,1)]$$
is the canonical embedding, we have the equality $(\zt^*_P,\id_\R)\circ \bar\zt=(\zt,\zp_\R)$ for projections
$$(\zt^*_P,\id_\R)\circ \bar\zt\,,(\zt,\zp_\R):P\ti\sT^*\R\to M\ti\R\,,$$
so $X^c_{\bar H}$ i projectable onto the vector field $X^c_{H_t}+\pa_t$ on $M\ti\R$ which is the autonomization of the time-dependent contact Hamiltonian vector field $X^c_{H_t}$.
In other words, the nonautonomous contact dynamics associated with the time-dependent contact Hamiltonian $H:P\ti\R\to\R$ on $P\ti\R$ is given by the time-dependent Hamiltonian contact vector field $X^c_{H_t}$ on $M$.
\begin{example}
For the trivial (or a local trivialization) principal bundle $\zt:P=\R^\ti\ti M\to M$ and the corresponding local coordinates $(s,y^a)$, $s\ne 0$, where $s$ is the standard coordinate on $\R^\ti\subset\R$ and $y=(y^a)$ are coordinates on $M$, we have $L^*_P=\R^*\ti M$ with coordinates $(z,y)$ and the embedding
$$i^*_P:\R^\ti\ti M\to\R^*\ti M\,,\quad i^*_P(s,y)=(s^{-1},y)\,.$$
The trivialization induces also a contact 1-form $\zh=i_{\pa_s}\zw$ on $M$, so that the symplectic form $\zw$ on $P$ reads
$$\zw(s,y)=\xd s\we\zh(y)+s\cdot\xd\zh(y)\,.$$
The projection $\bar\zt:\bar P\to\bar L^*_P=L^*_P\ti\R$ takes the form (cf. \ref{barzt})
$$\bar\zt(s,y,t,p)=(z=p/s,y,t)\,.$$
Any function $\hat H(y,t)$ (a reduced time-dependent contact Hamiltonian) on $M\ti\R$ induces a time-dependent contact Hamiltonian $H(s,y,t)=s\cdot\hat H(y,t)$ on $P\ti\R$ whose
autonomization is the function
$$ \bar H(s,y,t,p)=s\cdot \hat H(y,t)+p$$
on $\bar P=P\ti\sT^*\R$.
The corresponding Hamiltonian vector field $X_{\bar H}$ reads (cf. (\ref{cHvf}))
$$X_{\bar H}(s,y,t,p)=s\cdot\cR(\hat H_t)(y)\pa_s+X^c_{\hat H_t}(y)-s\cdot\frac{\pa\hat H}{\pa t}(y,t)\pa_{p}+\pa_t\,,$$
where $X^c_{\hat H_t}$ is uniquely determined by (cf. \ref{contact})
\begin{equation}\label{tdHvf}
i_{X^c_{\hat H_t}}\eta =-\hat H_t\,,\qquad i_{X^c_{\hat H_t}}\xd\eta =\xd\zs_t-\mathcal{R}(\hat H_t)\,\eta\,,
\end{equation}%
and $\cR$ is the Reeb vector field for the contact form $\zh$.
Since in coordinates $(s,y,t,p)$ on $\bar P$ we have $p=s\cdot z$, the Hamiltonian vector field $X_{\bar H}$ in these coordinates looks like
$$X_{\bar H}(s,y,t,z)=s\cdot\cR(\hat H_t)(y)\pa_s+X^c_{\hat H_t}(y)-\frac{\pa\hat H}{\pa t}(y,t)\pa_{z}+\pa_t\,,$$
so that
$$X^c_{\bar H}(z,y,t)=X^c_{\hat H_t}(y)-\frac{\pa\hat H}{\pa t}(y,t)\pa_{z}+\pa_t\,,$$
where $X^c_{\hat H_t}$ is uniquely determined by equations (\ref{tdHvf}).
Projecting once more, this time on $M\ti\R$, we get the autonomization $X^c_{\hat H_t}+\pa_t$
of the time-dependent contact Hamiltonian vector field $X^c_{\hat H_t}$. The corresponding contact Hamilton equations in Darboux coordinates for $\zh$ read
$$
\dot{q}^i= \frac{\partial \hat H_t}{\partial p_i}, \qquad \dot{p}_i = -\frac{\partial \hat H_t}{\partial q^i}-
p_i\frac{\partial \hat H_t}{\partial z}, \quad \dot{z} = p_i\frac{\partial \hat H_t}{\partial p_i} - \hat H_t.
$$
\end{example}

\medskip To present the autonomizations of time-dependent contact Hamiltonians in terms of sections of line bundles, let us observe that we may identify $L^*_{\bar P}$ with $\sV(L^*_P)\ti\R$ which is canonically a line bundle over $\bar L^*_P$,
$$\zt^*_{\bar P}:L^*_{\bar P}=\sV(L^*_P)\ti\R=(L^*_P\ti_ML^*_P)\ti\R\to L^*_P\ti\R=\bar L^*_P\,.$$
Here, for any vector bundle $E\to M$ we denote with $\sV E\subset\sT E$ the corresponding vertical bundle (vertical subbundle in $\sT E$), and with $\sV^*E$ its dual. It is well known that $\sV E=E\ti_ME$, so that $\sV^* E=E\ti_ME^*$.  Indeed, the line bundle $L^*_{\bar P}$ is the manifold of orbits of the $\Rt$-principal bundle $P\ti\sT^*\R\ti\R^*$ with the $\Rt$-action
$$\bar h_s(x,t,p,u)=(h_s(x),t,s\,p,s\,u)\,,$$
that immediately leads to this identification. Consequently, the line bundle $L_{\bar P}$ is
$$\zt_{\bar P}:L_{\bar P}=\sV^*(L^*_P)\ti\R=(L^*_P\ti_ML_P)\ti\R\to L^*_P\ti\R=\bar L^*_P\,,$$
and the embedding of $\bar P$ into $L_{\bar P}$ reads
\be\label{wsad}i_{\bar P}:P\ti\sT^*\R\to(L^*_P\ti_ML_P)\ti\R\,,\quad
i_{\bar P}(x,t,p)=\left(p\cdot i^*_P(x),i_P(x),t\right)\,.\ee
It is convenient to identify sections of the line bundle $L^*_{\bar P}$ with time-dependent sections of $\sV(L^*)$; to a time-dependent section $\zr:L^*_P\ti\R\to\sV(L^*)$ there corresponds the autonomous section $\zr^a$ of $L^*_{\bar P}$,
$$\zr^a(v_y,t)=(\zr(v_y,t),t)\,.$$
Note that the line bundle $\sV(L^*)$ has a distinguished section $\zm$ represented by the identity map $\zm(v)=(v,v)$, where $v\in L^*_P$.
With every time-dependent section $\zs:M\ti\R\to L^*_P$ we associate the time-dependent contact Hamiltonian
$$H_\zs:P\ti\R\to\R\,,\quad H_\zs(x,t)=H_{\zs_t}(x)=\zi_{\zs_t}(x)\,.$$
As we already know, the autonomization of this contact Hamiltonian is the autonomous contact Hamiltonian $\bar H_\zs$ on $\bar P$,
$$\bar H_\zs:\bar P=P\ti\sT^*\R\to\R\,,\quad \bar H_\zs(x,t,p)=H_\zs(x,t)+p\,.$$
Hence, $\bar H_\zs=H_{\bar\zs}$ for a time-dependent section $\bar\zs:\bar L^*_P\to\sV(L^*_P)$. The section $\bar\zs$ reads
$$\bar\zs(v_y,t)=\zm(v_y)+\zs(y,t)\,,$$
where $v_y\in(L^*_P)_y$.  Indeed, $\zi_{\bar\zs}$ as a linear function on $L_{\bar P}$ reads
$$\zi_{\bar\zs}(v_y,u_y,t)=\la u_y,\zs_t(y)+v_y\ran\,,$$
where $v_y\in(L^*_P)_y$ and $u_y\in (L_P)_y$\,, $y\in M$.
Using the form (\ref{wsad}) of the embedding $i_{\bar P}$, we get $\zi_{\bar\zs}$ restricted to $\bar P$ in the form
$$\zi_{\bar\zs}(x,t,p)=\la i_P(x),\zs_t(\zt(x))+p\cdot i^*_P(x)\ran=\zi_{\zs_t}(x)+p\cdot\la i_P(x),i^*_P(x)\ran=\bar H_\zs(x,t,p)\,.$$
Hence, we get the following.
\begin{theorem}
Let $\zs:M\ti\R\to L^*_P$ be a time-dependent section of the line bundle $\zt^*_P:L^*_P\to M$. Then  the Hamiltonian vector field $X_{\bar\zs}$ on the symplectic $\Rt$-principal bundle $\bar P=P\ti\sT^*\R$ described above projects under the obvious projection $(\zt,\zp_\R):P\ti\sT^*\R\to M\ti\R$ onto the autonomization $X^c_{\zs_t}+\pa_t$ of the time-dependent contact Hamiltonian vector field $X^c_{\zs_t}$ on $M$.
\end{theorem}
\noindent The time-dependent contact Hamiltonian vector field $X^c_{\zs_t}$ on $M$ we call \emph{induced by $\zs$}. In other words, the nonautonomous contact dynamics associated with the time-dependent section $\zs$ of $L^*_P$ is represented by the time-dependent contact Hamiltonian vector field $X_{\zs_t}$. Hence, the contact Hamilton equations for time-dependent sections of $L^*_P$ are formally the same as in the autonomous case, with the only difference that we derive the contact Hamiltonian vector fields $X^c_{\zs_t}$ for every fixed time parameter $t\in\R$.
\begin{example}\label{FE}
We will construct an example of a time-dependent contact Hamiltonian for the M\"obius band $B\to S^1$ which is the simplest model of a non-trivializable line bundle. More precisely, our contact manifold will be $\sJ^1B^*$ with the symplectic cover $\zt:\sT^*B^\ti\to\sJ^1B^*$ (cf. Example \ref{jb}).

\medskip\noindent Let us start with two charts $\cO_1$ and $\cO_2$  on $B$ with coordinates $(x,s)\in]0,2\pi[\ti\R$ and $(x',s')\in]\zp,3\zp[\ti\R$, respectively. The intersection $\cO_{12}=\cO_1\cap\cO_2$ consists of two disjoint sets $\cO_{12}=\cO_{12}^1\cup\cO_{12}^2$ represented by $]\zp,2\zp[\ti\R$ in $\cO_1$ and $\cO_2$, and $\cO_{12}^2$ represented by $]0,\zp[\ti\R$ in $\cO_1$ and with $]2\pi,3\zp[$ on $\cO_2$, with the transition functions
$(x',s')=(x,s)$ on $\cO_{12}^1$, and $(x',s')=(x+2\pi,-s)$ on $\cO_{12}^2$.
Since the transition function on $]\zp,2\zp[$ is the identity, it will be more convenient to look at $B$ as obtained from the trivial line bundle $L=]0,3\zp[\ti\R$, with global coordinates $(x,s)$, by gluing $]0,\zp[\ti\R$ with $]2\zp,3\zp[\ti\R$ by the map $(x,s)\mapsto(x+2\zp,-s)$. This is how we show the M\"obius band to children. In this picture $B^\ti$ is obtained in almost the same way, with the only difference that we admit only $s\ne 0$. Now, sections $\zs$ of $B$ are identified with functions $\zs:]0,3\zp[\to\R$ respecting the gluing, i.e., such that
$\zs(x+2\zp)=-\zs(x)$ for $x\in]0,\zp[$. This makes clear that every section must take the value 0, so the bundle $B$ is non-trivializable.

\medskip\noindent From this description of $B$, thus $B^\ti$, we obtain an analogous description of $\sT^*B^\ti$, being $\sT^*L^\ti\simeq]0,3\zp[\ti\R^3$, with global coordinates $(x,s,p,z)$, and subject of gluing by the map
$$(x,s,p,z)\mapsto(x+2\zp,-s,p,-z)\quad\text{for}\quad x\in]0,\zp[\,.$$
The cotangent bundle $\sT^*B^\ti$ is canonically a symplectic $\Rt$-principal bundle, with the $\Rt$-action (cf. Example \ref{pl}) represented by
$$s_0\cdot(x,s,p,z)=(x,s_0\cdot s,s_0\cdot p,z)$$ on $\sT^*L^\ti$
(this respects the gluing), and with the symplectic form $\zw$ represented by the canonical symplectic form
$$\zw=\xd x\we\xd p+\xd s\we\xd z$$ on $\sT^*\Lt$ (it is easy to see that this form is invariant with respect to the gluing).

\medskip\noindent Since we have two homogeneous coordinates of degree 1, namely $s$ and $p$, it is convenient to change the coordinates so, that we will have a manifestly trivial $\Rt$-principal bundle (with one homogeneous coordinate of degree 1). Like in Example \ref{jb}, we just replace the coordinate $p$ with the homogeneous coordinate $\p=p/s$ of degree 0, so that $(x,\p,z)$ are coordinates in $\sJ^1L^*$, representing coordinates in the corresponding two charts in $\sJ^1B^*$ \emph{via} gluing. In the new coordinates, the gluing in $\sT^*\Lt$ looks like
$$(x,s,\p,z)\mapsto(x+2\zp,-s,-\p,-z)\quad\text{for}\quad x\in]0,\zp[$$
and the symplectic form like (see Example \ref{jb})
\be\label{sf}\zw=\xd s\we(\xd z-\p\,\xd x)+s\cdot\xd x\we\xd\p\,.\ee
This symplectic form on $\sT*\Lt$ is invariant with respect to the gluing, so it represents a symplectic form on $\sT^*B^\ti$, and it is the symplectization of the contact form $\zh=\xd z-\p\,\xd x$ on $\sJ^1L^*$, that means that the canonical contact structure on $\sJ^1L^*$ is trivial. But this contact form is not invariant with respect to gluing, since the gluing map changes $\zh$ to $-\zh$. This is because the canonical contact structure on $\sJ^1B^*$ is non-trivializable and represented by $\zh$ and $-\zh$ in our two charts.

\medskip\noindent A 1-homogeneous Hamiltonian $H(x,s,\p,z)=s\cdot\hat H(x,\p,z)$ on $\sT^*\Lt$ represents a 1-homogeneous Hamiltonian on $\sT^*B^\ti$ if and only if it respects the gluing, i.e.,
$$s\cdot\hat H(x,\p,z)=-s\cdot\hat H(x+2\zp,-\p,-z)\quad\text{for}\quad x\in]0,\zp[.$$
In consequence, the reduced Hamiltonian $\hat H$ on $\sJ^1L^*$ should satisfy
\be\label{hh}\hat H(x,\p,z)=-\hat H(x+2\zp,-\p,-z)\quad\text{for}\quad x\in]0,\zp[.\ee
This also shows that such a $\hat H$ does not define any function (reduced contact Hamiltonian) on $\sJ^1B^*$, since it is not invariant with respect to the gluing in $\sJ^1L^*$. It defines functions in our two charts which differ by sign. This reflects the fact that, since the contact structure is non-trivializable, contact Hamiltonians for $\sJ^1B^*$ (i.e., 1-homogeneous functions on $\sT^*B^\ti$) cannot be represented by functions on $\sJ^1B^*$. This is exactly the difference with the standard approach to contact Hamiltonian mechanics. We can also consider time-dependent Hamiltonians for which $\hat H$ depends additionally on $t\in\R$, $\hat H_t(x,\p,z)=\hat H(x,\p,z,t)$, and each of $\hat H_t$ satisfies (\ref{hh}).

\medskip\noindent As a numerical example let us take
$$\hat H_t(x,\p,z)=\frac{\cos(x/2)}{2}\left(\p^2-z^2\right)+f(t)\sin(x/2)\p z,$$
where $f=f(t)$ is any function, which clearly satisfies (\ref{hh}) for each $t$. This is particularly interesting Hamiltonian, since for $f(t)\ne 0$ it is hyperregular in the contact sense (cf. \cite{Grabowska:2022a}), i.e., the Legendre map
$$(x,\p,z)\mapsto\left(x,\frac{\pa\hat H_t}{\pa \p}(x,\p,z),\frac{\pa\hat H_t}{\pa z}(x,\p,z)\right)$$
is a diffeomorphism. Indeed, for each $x$ the Legendre map is linear in coordinates $\p,z$ with the matrix
$$\begin{bmatrix}
\cos(x/2) & \ f(t)\sin(x/2)\\
f(t)\sin(x/2) & -\cos(x/2)\end{bmatrix}$$
which is clearly invertible.

\medskip To obtain the contact Hamiltonian vector field $X^c_{H_t}$, we calculate first the standard Hamiltonian vector field $X_{H_t}$ on $\sT^*L^\ti$, where $H_t(x,s,\p,z)={s\cdot\hat H_t(x,\p,z)}$. We have
\beas\xd(s\cdot\hat H_t)&=&\hat H_t\,\xd s+s\big(\cos(x/2)\p+f(t)\sin(x/2)z\big)\xd\p
+s\big(-\cos(x/2)z+f(t)\sin(x/2)\p\big)\\
&&+\frac{s}{4}\big(\sin(x/2)\left(z^2-\p^2\right)+2f(t)\cos(x/2)\p z\big)\xd x.
\eeas
This, together with the symplectic form (\ref{sf}), gives the Hamiltonian vector field
\beas X_{H_t}&=&s\Big(f(t)\sin(x/2)\p-\cos(x/2)z\Big)\frac{\pa}{\pa s}\\
&&+\Big(\cos(x/2)\p+f(t)\sin(x/2)z\Big)\frac{\pa}{\pa x} +\frac{\cos(x/2)}{2}\Big(\p^2+z^2\Big)\frac{\pa}{\pa z}\\
&&-\bigg(\frac{\sin(x/2)}{4}\Big(z^2+\big(4f(t)-1\big)\p^2\Big)
+\frac{\cos(x/2)}{2}\Big(f(t)-2\Big)\p z\bigg)\frac{\pa}{\pa\p}\,.
\eeas
As it should be, the vector field $X_{H_t}$ is projectable onto $\sJ^1L^*$ giving the contact Hamiltonian vector field
\beas X^c_{H_t}&=&\Big(\cos(x/2)\p+f(t)\sin(x/2)z\Big)\frac{\pa}{\pa x} +\frac{\cos(x/2)}{2}\Big(\p^2+z^2\Big)\frac{\pa}{\pa z}\\
&&-\bigg(\frac{\sin(x/2)}{4}\Big(z^2+\big(4f(t)-1\big)\p^2\Big)
+\frac{\cos(x/2)}{2}\Big(f(t)-2\Big)\p z\bigg)\frac{\pa}{\pa\p}\,.
\eeas
One can immediately see that, indeed, this vector field respects the gluing $(x,\p,z)\mapsto(x+2\zp,-\p,-z)$, so it represents the corresponding time-dependent contact Hamiltonian vector field on $\sJ^1B^*$.
\end{example}
\begin{remark}
We should note that an interesting time-dependent contact dynamics is described in \cite{deLeon:2022} with the use of the concept of a \emph{cocontact manifold}, with the extended contact manifolds $\R\ti\sT^*\R\ti\R$ as working examples. The difference with our approach is that these are not the time-dependent Hamiltonians in our sense, but \emph{cocontact Hamiltonians}. However, it seems that, at least for trivial cocontact manifolds, the resulted time-dependent dynamics coincides with ours. In any case, a deeper comparison of the two approaches deserves further studies.
\end{remark}
\section{Contact Hamilton-Jacobi theory}
\subsection{Classical Hamilton-Jacobi equations}
In our understanding, the geometrical essence of the Hamilton-Jacobi theory is based on the following trivial observation.
\begin{theorem}\label{HJ}
Let $(P,\zw)$ be a symplectic manifold, $H:P\to\R$ be a Hamiltonian function, and $\cL\subset P$ be a Lagrangian submanifold. Then  $H\,\big|_\cL$ is closed (locally constant) if and only if the Hamiltonian vector field $X_H$ is tangent to $\cL$.
\end{theorem}
\begin{proof}
For any $x\in\cL$ and any $Y\in\sT_x\cL$ we have
$$\la\xd H(x),Y\ran=\zw_x(X_H(x),Y)\,,$$
so that $\xd H(x)$ acts trivially on $\sT_x\cL$ if and only if $X_H(x)$ is tangent to the Lagrangian submanifold $\cL$.

\end{proof}
\noindent In standard applications, $P=\sT^*Q$ and the Lagrangian submanifold is taken to be $\cL(S)=\xd S(Q)$ for a function $S:Q\to\R$. Then  the above observation means that $\xd(H\circ\xd S)=0$ is equivalent to the fact that the trajectories of $X_H$ starting from points of $\xd S(Q)$ lie entirely in $\xd S(Q)$, thus project onto trajectories of the projected vector field $(\zp_Q)_*\left(X_H\,\big|_{\xd S(Q)}\right)$, where $\zp_Q:\sT^*Q\to Q$ is the canonical projection (note that $\zp_Q\,\big|_{\cL(S)}:\xd S(Q)\to Q$
is a diffeomorphism). We call the equation $\xd(H\circ\xd S)=0$, i.e.,
$$H(q^i,\frac{\pa S}{\pa q^j})=const\,,$$
the \emph{Hamilton-Jacobi equation}.

\medskip\noindent
The time-dependent (nonautonomous) version of the Hamilton-Jacobi theory is essentially based on the same principle. Like in the previous section, we associate with a time-dependent Hamiltonian $H:P\ti\R\to\R$ on the symplectic manifold $(P,\zw)$ the autonomous Hamiltonian
$$\bar H(x,t,p)=H(x,t)+p$$
on the symplectic manifold $\bar P=P\ti\sT^*\R$ (with the symplectic 2-form $\bar\zw=\zw+\xd t\we\xd p$), where $(t,p)$ are the standard Darboux coordinates on $\sT^*\R$, and we apply Theorem \ref{HJ} to $\bar H$ and a Lagrangian submanifold $\bar\cL\subset\bar P$.
If $\bar H$ is locally constant on $\bar\cL$, say $\bar H(x,t,p)=H(x,t)+p=c$, then replacing $p$ with $p+c$ (note that $(t,p+c)$ are also global Darboux coordinates on $\sT^*\R$), we can always assume that $\bar H$ is zero on $\bar\cL$; this change of coordinates does not change the Hamiltonian vector field. The condition $\bar H\,\big|_{\bar\cL}=0$ we call the (nonautonomous) \emph{Hamilton-Jacobi equation}. As we already know, it is equivalent to the fact that  the Hamiltonian vector field,
$$X_{\bar H}(x,t,p)=X_{H_t}(x)-\frac{\pa H}{\pa t}(x,t)\pa_p+\pa_t\,,$$
is tangent to $\bar\cL$. Moreover, as on $\bar\cL$ the momentum $p$ is a function of $(x,t)$, the Lagrangian submanifold projects diffeomorphically with respect to the canonical projection $\bar\zp:P\ti\sT^*\R\to P\ti\R$ onto a submanifold $\cL$ of $P\ti\R$ and the Hamiltonian vector field $X_{\bar H}$ projects onto the vector field $X_H(x,t)=X_{H_t}(x)+\pa_t$ which is the autonomization of the time-dependent vector field $X_{H_t}$ on $P$. Hence, the trajectory of $X_H$ with the initial point $(x_0,t_0)$ projects \emph{via} the canonical projection $P\ti\R\to\R$ ono the trajectory $\zg:(a,b)\to P$ of the time-dependent vector field $X_{H_t}$ with the initial condition $\zg(t_0)=x_0$.

\medskip\noindent In the classical situation, when $P=\sT^*Q$ with the canonical symplectic form $\zw_Q=\xd q^i\we\xd p_i$ and $\bar\cL(S)=\xd S(Q\ti\R)$ for a function $S:Q\ti\R\to\R$, we have the corresponding \emph{nonautonomous Hamilton-Jacobi equation} in the form $\bar H\circ\xd S=0$, i.e.,
$$ H\left(q^i,\frac{\pa S}{\pa q^j},t\right)+\frac{\pa S}{\pa t}=0\,.$$
In this case the Hamiltonian vector field $X_{\bar H}$ is tangent to $\bar\cL(S)$, so that the trajectories of $X_{\bar H}$ starting from points of $\bar\cL(S)$ project \emph{via}
the canonical projection $\zp_{Q\ti\R}:\sT^*(Q\ti\R)\to Q\ti\R$ onto trajectories of the projected vector field
$$\ul{X}_{\bar H}=(\zp_{Q\ti\R})_*\left(X_{\bar H}\,\big|_{\bar\cL(S)}\right)$$
on $Q\ti\R$. As
$$\ul{X}_{\bar H}(q,t)=(\zp_Q)_*\left(X_{H_t}(q,\xd S_t(q)\right)+\pa_t$$
is the autonomization of the time-dependent  vector field
$$\ul{X}_{H_t}(q)=(\zp_Q)_*\left(X_{H_t}(q,\xd S_t(q))\right)$$
on $Q$, we get the reduction of the Hamilton equations on $\bar\cL(S)=\xd S(Q\ti\R)$ onto a system of time-dependent first-order ordinary differential equations on $Q$, represented by the time-dependent vector field $\ul{X}_{H_t}$. This gives the following classical \emph{Jacobi Theorem}.
\begin{theorem}
Let $H:\sT^*Q\ti\R\to\R$ be a nonautonomous Hamiltonian on a symplectic manifold $P=\sT^*Q$ and
let $X_{\bar H}$ be the Hamiltonian vector field on $\sT^*(Q\ti\R)$ associated with the autonomization
$$\bar H:\bar P=\sT^*(Q\ti\R)=P\ti\sT^*\R\to\R\,,\quad \bar H(x,t,p)=H(x,t)+p$$
of $H$. Take a function $S:Q\ti\R\to\R$, so that $\xd S:Q\ti\R\to\sT^*(Q\ti\R)$ maps $Q\ti\R$ onto the Lagrangian submanifold $\bar\cL(S)=(\xd S)(Q\ti\R)$ of $\bar P=\sT^*(Q\ti\R)$. Then  the canonical projection $\zp_{Q\ti\R}:\sT^*(Q\ti\R)\to Q\ti\R$ maps diffeomorphically the Lagrangian submanifold $\bar\cL(S)$ onto $Q\ti\R$, so the vector field $X_{\bar H}\,\big|_{\bar\cL(S)}$ onto a vector field $\ul{X}_{\bar H}$ on $Q\ti\R$, and the followings are equivalent:
\begin{enumerate}
\item the nonautonomous Hamilton-Jacobi equation
$$H\left(q,\frac{\pa S}{\pa q},t\right)+\frac{\pa S}{\pa t}=0$$
is satisfied;
\item the Hamiltonian vector field $X_{\bar H}$ is tangent to $\bar\cL(S)$;
\item the vector fields $\ul{X}_{\bar H}$ and $X_{\bar H}$ are $\xd S$-related;
\item the section $\xd S:Q\ti\R\to\sT^*(Q\ti\R)$ maps trajectories of the vector field $\ul{X}_{\bar H}$ onto trajectories of the Hamiltonian vector field $X_{\bar H}$.
\end{enumerate}
\end{theorem}

\subsection{Contact Hamilton-Jacobi equations}
Let now $M$ be a contact manifold with the contact structure $(P,\zt,M,\zw,h)$, where $\zt:P\to M$
is an $\R^\ti$-principal bundle with the $\R^\ti$-action $h:\R^\ti\ti P\to P$ and $\zw$ is a 1-homogeneous symplectic form on $P$, $h_t^*\zw=t\cdot\zw$. Let us choose a Legendre submanifold $\cL_0\subset M$, so that $\cL=\zt^{-1}(\cL_0)$ is an $\R^\ti$-invariant Lagrangian submanifold of $P$. According to Theorem \ref{HJ}, a Hamiltonian $H:P\to\R$ is locally constant on $\cL$ and  1-homogeneous if and only if its Hamiltonian vector field $X_H$ is tangent to $\cL$. Since $\cL$ is a union of fibers of $\zt:P\to M$ over $\cL_0$, any locally constant 1-homogeneous Hamiltonian on $\cL$ must be 0 on $\cL$. The vector field $X_H$ is projectable and tangent to $\cL$, so that $\zt_*(X_H)=X^c_H$ is tangent to $\cL_0$. Conversely, if $X^c_H$ is tangent to $\cL_0$, then clearly the projectable vector field $X_H$ must be tangent to $\cL=\zt^{-1}(\cL_0)$, thus $H\,\big|_\cL=0$ and we get a contact version of Theorem \ref{HJ}, the \emph{contact Jacobi theorem}.
\begin{theorem}\label{cHJ}
Let $\cL_0$ be a Legendre submanifold of a contact manifold $M$ with the contact structure $(P,\zt,M,\zw,h)$. A 1-homogeneous Hamiltonian $H$ on $P$ equals 0 on the submanifold $\cL=\zt^{-1}(\cL_0)$ if and only if the Hamiltonian vector field $X^c_H$ is tangent to $\cL_0$.
\end{theorem}
\noindent Since 1-homogeneous Hamiltonians are in a one-to-one correspondence with sections $\zs$ of the line bundle $\zt^*_P:L^*_P\to M$, we can write $H=H_\zs=\zi_\zs$ for some section $\zs$. The vanishing of $H_\zs$ on $\zt^{-1}(\cL_0)$ is equivalent to the fact that $\zs$ vanishes on $\cL_0\subset M$.
In the trivial case there is a one-to-one correspondence between 1-homogeneous (contact) Hamiltonians $H:P\to\R$ and functions $\hat H:M\to\R$ which we call reduced contact Hamiltonians. In local coordinates,
$$H(s,y)=s\cdot \hat H(y)\,.$$
The Hamiltonian vector field $X_H$ on $P$ projects \emph{via} $\zt:P\to M$ onto the contact Hamiltonian vector field $X^c_{\hat H}$ on $M$.

\medskip As an example,
consider the canonical trivial contact manifold $M=\sJ^1(Q;\R)=\R\ti\sT^*Q$ and  let  $X^c_{\hat H}$ be the contact Hamiltonian vector field on $M$ associated with a reduced contact Hamiltonian $\hat H:\R\ti\sT^*Q\to\R$ and described in (\ref{con-dyn}). Consider a function $S:Q\to\R$ and the Legendre submanifold $\cL_0(S)$ in $\R\ti\sT^*Q$ being the image of  $\sj^1(S):Q\to\R\ti\sT^*Q$, i.e., $\cL_0(S)=\sj^1(S)(Q)$. In local coordinates,
$$\cL_0(S)=\big\{(S(q),\xd S(q))\,\big|\,q\in Q\big\}=\bigg\{\left(S(q),q^i,\frac{\pa S}{\pa q^j}(q)\right)\,\big|\,q\in Q\bigg\}\,.$$
The canonical projection $\zp^1_Q:\sJ^1(Q,\R)=\R\ti\sT^*Q\to Q$ maps the Legendre submanifold $\cL_0(S)$ diffeomorphically onto $Q$, so the contact Hamiltonian vector field along $\cL_0(S)$, $X^c_{\hat H}\,\big|_{\cL_0(S)}$, onto a vector field $\ul{X}^c_{\hat H}$ on $Q$.
Theorem \ref{cHJ} implies now immediately the following version of \emph{contact Jacobi Theorem}.
\begin{theorem}
The vector fields $\ul{X}^c_{\hat H}$ and $X^c_{\hat H}$ are $\sj^1(S)$-related if and only if
$$\hat H\circ\sj^1(S)(q)=\hat H\left(S(q),q^i,\frac{\pa S}{\pa q^j}(q)\right)=0$$
for all $q\in Q$.
\end{theorem}
\noindent The PDE
$$\hat H\circ\sj^1(S)=0$$
we call the \emph{contact Hamilton-Jacobi equation}.
\begin{remark}
The above contact Hamilton-Jacobi equations for extended cotangent bundle reduce to the ones obtained in \cite{deLeon:2021c,Esen:2021a}. Below we extend them to the case of first jet bundles $\sJ^1(L)$ for arbitrary line bundles.
\end{remark}

\medskip\noindent In the above example the contact manifold was the bundle of first jets of sections of the trivial line bundle $\R\ti Q\to Q$ with its canonical contact 1-form. We can easily generalize all this to the case of an arbitrary line bundle.
As we already know, 1-homogeneous contact Hamiltonians $H:P\to\R$ on the symplectic $\Rt$-bundle $P=\sT^*(\Lt)$  are in a one-to-one correspondence with sections $\zs:\sJ^1(L^*)\to L^*_P$ of the line bundle $L^*_P=\sJ^1(L^*)\ti_QL^*$; the Hamiltonian $\zi_\zs$ associated with $\zs$ we denote also $H_\zs$. Locally everything looks like in the case of $J^1(Q,\R)$, so the contact Jacobi theorem takes the following form.
\begin{theorem}\label{cHJ1}
Let $\zt_0:L\to Q$ be a line bundle, $\zt^1(L^*):\sJ^1(L^*)\to Q$ be the bundle of first jets of sections of the dual bundle $\zt^*_0:L^*\to Q$ equipped with the canonical contact structure $$\left(P=\sT^*(L^\ti),\zt,\sT^*h,\zw_{L^\ti}\right)\,,$$
where $h$ is the canonical action of $\R^\ti$ on the $\Rt$-principal bundle $L^\ti$. If $$\zs:\sJ^1(L^*)\to L^*_P=\sJ^1(L^*)\ti_QL^*$$
is a section of the line bundle $L^*_P$, $X^c_\zs$ is the contact Hamiltonian vector field on $\sJ^1(L^*)$ associated with the contact Hamiltonian $H_\zs$, and $S:Q\to L^*$ is a section of the line bundle $L^*\to Q$, then the following are equivalent:
\begin{enumerate}
\item we have \be\label{cHJe} \zs\circ\sj^1(S)=0\,;\ee
\item the contact Hamiltonian vector field  $X^c_\zs$ restricted to the Legendre submanifold $\cL_0(S)=\sj^1(S)(Q)$ in $\sJ^1(L^*)$ and its projection $\ul{X}^c_\zs=({\zt^1(L^*)})_*(X^c_\zs\,\big|_{\cL_0})$ onto $Q$ are
$\zt^1(L^*)$-related.
\end{enumerate}
\end{theorem}
\noindent The equation (\ref{cHJe}) we call the \emph{contact Hamilton-Jacobi equation}. Note that the section $\zs$ can be identified with a morphism $\tilde\zs:\sJ^1(L^*)\to L^*$ of fiber bundles over $Q$ defined by $\tilde\zs=pr_2\circ\zs$, where
$$pr_2:\sJ^1(L^*)\ti_QL^*\to L^*$$
is the projection onto the second factor.
\subsection{Contact Hamilton-Jacobi equations -- the nonautonomous case}
In Section 5.1 we described the autonomization of a contact structure and the corresponding contact Hamiltonian dynamics. Let $(P,\zt,M,\zw,h)$ be a contact structure and
$(\bar P,\bar\zt,\bar M,\bar\zw,\bar h)$, where $\bar P=P\ti\sT^*\R$, be its autonomization.

Let us consider now a time-dependent section $\zs:M\ti\R\to L^*_P$ and the corresponding time-dependent 1-homogeneous contact Hamiltonian $H_\zs:P\ti\R\to\R$, so that every Hamiltonian $H_{\zs_t}:P\to\R$ is 1-homogeneous. Note that we can identify $\zs$ with a section ${\tilde\zs}$ of the line bundle
$$\bar\zt^*_P=(\zt^*_P,\id_\R):\bar L^*_P=L^*_P\ti\R\to M\ti\R\,,\quad {\tilde\zs}(y,t)=(\zs(y,t),t)\,.$$
The autonomization of $H_\zs$ is the contact Hamiltonian $\bar H_\zs:\bar P\to\R$ of the form $\bar H_\zs(x,t,p)=H_\zs(x,t)+p$ and corresponds to the autonomization $\bar\zs$ of the time-dependent section $\zs$,
$$\bar\zs:\bar L^*_P=L^*_P\ti\R\to \sV(L^*_{P})=(L^*_P\ti_M L^*_P)\,,\quad \bar\zs(v_y,t)=(v_y,v_y+\zs(y,t))\,,$$
i.e., $\bar H_\zs=H_{\bar\zs}$. The corresponding Hamiltonian vector field $X_{\bar\zs}$ on $\bar P$  reads
$$X_{\bar\zs}(x,t,p)=X_{{\zs_t}}(x)-\frac{\pa H_\zs}{\pa t}(x,t)\pa_{p}+\pa_t$$ and
projects onto the contact Hamiltonian vector field $X^c_{\bar\zs}$ on $\bar L^*_P=L^*_P\ti\R$.
If $P=\R^\ti\ti M$ is a local trivialization of the $\R^\ti$-principal bundle $\zt:P\to M$ with coordinates $(s,y)$ and the corresponding contact 1-form $\zh=\zh(y)$ on $M$, then we have the corresponding local trivialization $\bar P=\R^\ti\ti M\ti\sT^*\R$ of $\bar P$ (with coordinates $(s,y,t,p)$) and the corresponding local trivialization $\bar L^*_P=\R^*\ti M\ti\R$ of $\bar L^*_P$ (with coordinates $(\p,y,t)$), so that $\bar\zt:\bar P\to\bar L^*_P$ reads
$\bar\zt(s,y,t,p)=(p/s,y,t)$, the corresponding contact 1-form $\bar\zh$ on $\bar L^*_P$ reads $\bar\zh(\p,y,t)=\zh(y)-\p\,\xd t$, and the section $\bar\zs:\bar L^*_P\to \sV(L^*_{P})$ reads $\bar\zs(\p,y,t)=\zs(y,t)+\p$. The contact Hamiltonian vector field on $\bar L^*_P$ is
$$X^c_{\bar\zs}(\p,y,t)=X^c_{\zs_t}-\frac{\pa\zs}{\pa t}(y,t)\pa_{\p}+\pa_t\,.$$
Now, we can use the autonomous contact Hamilton-Jacobi theory which implies that, for any Legendre submanifold of $\cL_0$ of $\bar L^*_P$, the section $\bar\zs$ vanishes on $\cL_0$ if and only if $X^c_{\bar\zs}$ is tangent to $\cL_0$. In such a case $\p=-\zs(y,t)$
on $\cL_0$ and thus the canonical projection $\bar\zt^*_P:\bar L^*_P\to M\ti\R$ maps diffeomorphically $\cL_0$ onto a submanifold $\ul{\cL}_0$ of $M\ti\R$ and hence relates the contact Hamiltonian vector field $X^c_{\bar\zs}\,\big|_{\cL_0}$ along $\cL_0$ with its projection $$\ul{X}^c_{\bar\zs}=(\bar\zt^*_P)_*(X^c_{\bar\zs}\,\big|_{\cL_0})$$
onto $\ul{\cL}_0$.

\medskip\noindent
A particularly interesting is the case of the canonical contact structures on first jet bundles of line bundles, so let us take a line bundle $\zt_0:L\to Q$, and $P=\sT^*(L^\ti)$ which carries the canonical contact structure of $\sJ^1(L^*)$, where $\zt^*_0:L^*\to Q$ is the line bundle dual to $L$. We have then
$$\bar P=\sT^*(L^\ti)\ti\sT^*\R=\sT^*\left(L^\ti\ti\R\right)=\sT^*(\bar L^\ti)\,,$$
where $\bar L$ is the line bundle
$$\bar\zt_0=(\zt_0,\id_\R):\bar L=L\ti\R\to Q\ti\R\,,$$
so that $\bar\zt:\bar P\to\sJ^1(\bar L^*)$, where $\bar\zt^*_0:\bar L^*=L^*\ti\R\to Q\ti\R$ is the dual of $\bar L$. According to (\ref{fjs}),
$$\zt^*_P:L^*_P=\sJ^1(L^*)\ti_QL^*\to\sJ^1(L^*)\,,$$
so that $\sV(L^*_P)=L^*_P\ti_QL^*\to L^*_P$.
Note the canonical isomorphism
$$\Psi:\sJ^1(\bar L^*)\to\left(\sJ^1(L^*)\ti_QL^*\right)\ti\R=L^*_P\ti\R\,,\quad \Psi\left(\sj^1(\tilde S)(q,t)\right)=\left(\sj^1(S_t)(q),\frac{\pa S}{\pa t}(q,t),t\right)\,,$$
where $S_t(q)=S(q,t)$ is a section of $L^*$ and $\frac{\pa S}{\pa t}(q,t)\in\sV_{S(q,t)}L^*$ is the
vector tangent to $L^*$ at $S(q,t)$, represented by the vertical curve $S_q(t)=S(q,t)$ and identified with an element of $L^*_q$. Of course, $\sj^1$ refers the first jet of a section.

\medskip\noindent Let us choose a time-dependent section
$$\zs:\sJ^1(L^*)\ti\R\to\sJ^1(L^*)\ti_QL^*=L^*_P$$
of $\zt^*_P$, and let $X^c_{\bar\zs}$ be the corresponding contact Hamiltonian vector field on $\sJ^1(\bar L^*)$. The section $\zs$ is represented by a time-dependent morphism $\zs_0:\sJ^1(L^*)\ti\R\to L^*$ of fiber bundles over $Q$ which covers the identity on $Q$,  $\zs(y,t)=(y,\zs_0(y,t))$. The autonomization
$$\bar\zs:\sJ^1(\bar L^*)=L^*_P\ti\R\to L^*_P\ti_Q L^*\,,\quad \bar\zs(w,t)=(w,\bar\zs_0(w,t))\,,$$
in turn, is represented by the morphism $\bar\zs_0$ of fiber bundles over $Q$,
$$\bar\zs_0:\sJ^1(\bar L^*)\to L^*\,,$$ which covers the identity on $Q$,
\be\label{ee1}\bar\zs_0\left(\sj^1(\tilde S)(q,t)\right)=\zs_0\bigg(\sj^1(S_t)(q),t\bigg)+\frac{\pa S}{\pa t}(q,t)\,.\ee
Any time-dependent section $S:Q\ti\R\to L^*$ (i.e., a section $\tilde S(q,t)=(S(q,t),t)$ of $\bar L^*$) generates a  Legendre submanifold $\cL_0(S)$ of $\sJ^1(\bar L^*)$ being the image of $\sj^1(\tilde S):Q\ti\R\to \sJ^1(\bar L^*)$,
$$\cL_0(S)=\sj^1(\tilde S)(Q\ti\R)\,.$$
Let $X^c_{\bar\zs}$ be the contact Hamiltonian vector field on $\sJ^1(\bar L^*)$ corresponding to the autonomization of the section $\zs$. Since the canonical projection $\zt^1(\bar L^*):\sJ^1(\bar L^*)\to Q\ti\R$ maps the submanifold $\cL_0(S)$ diffeomorphically onto $Q\ti\R$, it maps the vector field $X^c_{\bar\zs}$ along $\cL_0(S)$ onto a vector field $\ul{X}^c_{\bar\zs}$ on $Q\ti\R$. As now, after the autonomization, we are working with the autonomous case, the autonomous Jacobi Theorem \ref{cHJ1} implies the following \emph{nonautonomous contact Jacobi Theorem}.
\begin{theorem} The following are equivalent.
\begin{enumerate}
\item the map $\bar\zs$ vanishes on $\cL_0(S)$, i.e. (cf. (\ref{ee1})),
\be\label{najhje}
\zs_0\bigg(\sj^1(S_t),t\bigg)+\frac{\pa S}{\pa t}=0\,;
\ee
\item the contact Hamiltonian vector field $X^c_{\bar\zs}$ is tangent to the Legendre submanifold $\cL_0(S)$ of $\sJ^1(\bar L^*)$;
\item The vector fields $\ul{X}^c_{\bar\zs}$ and ${X}^c_{\bar\zs}$ are $\sj^1(\tilde S)$-related.
\end{enumerate}
\end{theorem}
The PDE (\ref{najhje}) is our \emph{nonautonomous contact Hamilton-Jacobi equation}. For some particular $\zs_0$ we obtain various variants of Hamilton–Jacobi equations considered already in the literature, e.g. \emph{discounted Hamilton-Jacobi equations} which have been treated recently in \cite{Cannarsa:2020,Jin:2000}, where a relationship between the Herglotz variational principle and the Hamilton–Jacobi equation was established, and in \cite{Davini:2016,Davini:2021}, where the problem of the convergence of solutions was studied.
\begin{example}
For the trivial line bundle $L=\R\ti Q$ and $P=\sT^*(\Rt\ti Q)=\Rt\ti\R^*\ti\sT^*Q$, we have $L^*=\R^*\ti Q$, $\sJ^1(L^*)=\R^*\ti\sT^*Q$, $\bar L^*=\R^*\ti Q\ti\R$, $L^*_P=\sJ^1(L^*)\ti\R^*$, and $\sJ^1(\bar L^*)=\sJ^1(L^*)\ti\R^*\ti\R$. Starting from local coordinates $(q^i)$ on $Q$, we will use on the latter objects the coordinates $(z,q^i)$, $(z,q^i,p_j)$, $(z,q^i,t)$, $(z,q^i,p_j,p)$, and $(z,q^i,p_j,p,t)$, respectively. Moreover, we identify $L^*_P\ti\R$ with $\sJ^1(\bar L^*)$ as line bundles over $\sJ^1(L^*)\ti\R$. A time-dependent section $\zs$ of $L^*_P$ is represented by a function $H_0$ on $\sJ^1(L^*)\ti\R$,
$$\zs(z,q_i,p_j,t)=\left(z,q_i,p_j,H_0(z,q_i,p_j,t)\right)\,.$$
The autonomization $\bar\zs$ as a section of $\sV(L^*_P)=\sJ^1(L^*)\ti\R^*\ti\R^*$ is represented by the function $H=H_0+p$ on $\sJ^1(\bar L^*)=\sJ^1(L^*)\ti\R^*\ti\R$,
$$\bar\zs(z,q^i,p_j,p,t)=\left(z,q^i,p_j,p,t,H_0(z,q_i,p_j,t)+p\right)\,.$$
If we choose a Legendre submanifold in $\sJ^1(\bar L^*)$ to be $\sj^1(S)$ for a function
$S:Q\ti\R\to\R^*$, viewed as a section of $\bar L^*$, then the corresponding contact Hamilton-Jacobi equation takes the form
$$H_0\left(S(q,t),q_i,\frac{\pa S}{\pa q^j},t\right)+\frac{\pa S}{\pa t}=0\,.$$
In the case when $H_0(z,q_i,p_j,t)=H_1(q_i,p_j)+\zl\,t\,z$, we get
$$H_1\left(q_i,\frac{\pa S}{\pa q^j},t\right)+\zl\,S=0\,,$$
which is a time dependent version of the discounted Hamilton-Jacobi equation.
For a `free' Hamiltonian $H_1$ on $\sT^*Q$, $H_1(q^i,p_j)=\frac{1}{2}\sum_jp_j^2$, we get the nonlinear PDE
$$\zl\,S+\frac{1}{2}\sum_i\left(\frac{\pa S}{\pa q^i}\right)^2=0\,.$$
\end{example}
\section{Concluding remarks}
We proposed a novel approach to contact Hamiltonian mechanics which serves for contact structures of all kinds, contrary to the trend dominating in the existing literature to work exclusively with trivial (cooriented) contact structures. The Hamiltonian vector fields there are defined by means of Hamiltonians being functions on the contact manifold $M$ and the Reeb vector field for the given global contact form. This approach is non-geometric in the context of general (non-trivial) contact structures, since the obtained contact Hamiltonian vector fields strongly depend on the contact form chosen from an equivalence class. Our approach is completely intrinsic and free from this deficiency for the price that Hamiltonians live not on $M$ but on a principal bundle $P\to M$ with the structure group $\Rt=\GL(1,\R)$, equipped additionally with a homogeneous symplectic form. In this sense, contact geometry is not an odd-dimensional cousin of symplectic geometry, but a particular (homogeneous) symplectic geometry. This reminds the Kaluza-Klein theory, which is a classical unified field theory of gravitation and electromagnetism built around the idea of a fifth dimension beyond the common 4D of the space-time approach.

\medskip\noindent
We developed also a contact Hamilton-Jacobi theory which produces contact Hamilton-Jacobi equations which can be written in particulary compact and simple form in the case of canonical contact structures on the first jet bundles $\sJ^1(L)$ of sections of line bundles $L$. These contact structures can be characterized as linear and play in contact  mechanics a fundamental r\^ole, analogous to the r\^ole played by cotangent bundles in symplectic geometry. It is also interesting that particular cases of our contact Hamilton-Jacobi equations have been already studied in the literature as \emph{ad hoc} variants of the standard ones.

\medskip\noindent
The next step in our investigations in contact mechanics will be the Lagrangian side of the picture. We intend to include singular Lagrangians and follow therefore the powerful idea of \emph{Tulczyjew triple} \cite{Tulczyjew:1974,Tulczyjew:1977,Tulczyjew:1999} (some attempts in the case of \emph{extended tangent bundles} $\sT Q\ti\R$ one can find in \cite{Esen:2021}). Of course, non-trivial contact structures should be included, so one has to consider more sophisticated objects than just extended tangent bundles $\sT Q\ti\R$. As the approach \emph{via} Tulczyjew triples allow for immediate generalizations in the form of geometric mechanics on Lie algebroids, i.e., for linear Poisson structures (see e.g. \cite{Bruce:2015,Grabowska:2008,Grabowska:2011,Grabowska:2006,Grabowska:2016} for particle mechanics and \cite{Bruce:2016,Grabowska:2013,Grabowska:2015} for field theories), we believe that
an analogous theory can be developed for linear Jacobi structures. A relevant paper is in progress.

\small{\vskip1cm}

\noindent Katarzyna GRABOWSKA\\
Faculty of Physics \\
University of Warsaw\\
Pasteura 5,
02-093 Warszawa, Poland
\\Email: konieczn@fuw.edu.pl\\

\noindent Janusz GRABOWSKI\\ Institute of
Mathematics\\  Polish Academy of Sciences\\ \'Sniadeckich 8, 00-656 Warszawa, Poland
\\Email: jagrab@impan.pl \\

\end{document}